\documentclass[11pt,a4paper,oneside,reqno]{amsart}

\usepackage{amsmath, amsthm, amsopn, amssymb}
\usepackage{mathtools, thmtools, thm-restate}
\usepackage{bbm}
\usepackage{comment}
\usepackage{enumerate}
\usepackage{enumitem}
\usepackage{hyperref}
\usepackage{cleveref}

\usepackage{stmaryrd} 
\usepackage{xcolor}

\usepackage{geometry}
\usepackage{setspace}
\declaretheorem[name=Theorem,refname={Theorem,Theorems}]{HCL}

\declaretheorem[name=Theorem,parent=section]{theorem}
\declaretheorem[name=Lemma,sibling=theorem]{lemma}
\declaretheorem[name=Corollary,sibling=theorem]{corollary}

\declaretheorem[name=Proposition,sibling=theorem]{prop}
\declaretheorem[name=Question,sibling=theorem]{question}

\declaretheorem[name=Remark,style=remark,numbered=no]{remark}

\newcommand{\bR}{\mathbf{R}}

\newcommand{\RR}{\mathbb{R}}

\newcommand{\cA}{\mathcal{A}}

\newcommand{\cC}{\mathcal{C}}

\newcommand{\cF}{\mathcal{F}}
\newcommand{\cG}{\mathcal{G}}
\newcommand{\cH}{\mathcal{H}}
\newcommand{\cI}{\mathcal{I}}

\newcommand{\cS}{\mathcal{S}}

\newcommand{\cX}{\mathcal{X}}

\newcommand{\tC}{\tilde{C}}

\renewcommand{\Pr}{\mathbb{P}}
\newcommand{\Ex}{\mathbb{E}}

\newcommand{\Ber}{\mathrm{Ber}}
\newcommand{\1}{\mathbbm{1}} 
\newcommand{\DKL}{D_{KL}}
\newcommand{\Inf}{\mathrm{Inf}}

\newcommand{\br}[1]{\llbracket{#1}\rrbracket}

\newcommand{\<}{\langle}
\renewcommand{\>}{\rangle}

\renewcommand{\le}{\leqslant}
\renewcommand{\ge}{\geqslant}

\newcommand{\Q}{\mathbb{Q}}
\newcommand\R{\mathbb{R}}
\newcommand\Z{\mathbb{Z}}

\renewcommand{\leq}{\leqslant}
\renewcommand{\geq}{\geqslant}
\renewcommand{\le}{\leqslant}
\renewcommand{\ge}{\geqslant}
\renewcommand{\to}{\rightarrow}

\def\1{\mathbbm{1}}

\author{Marcelo Campos}
\address{Trinity College, University of Cambridge, Cambridge CB2 1TQ, United Kingdom}
\email{mc2482@cam.ac.uk}

\author{Wojciech Samotij}
\address{School of Mathematical Sciences, Tel Aviv University, Tel Aviv 6997801, Israel}
\email{samotij@tauex.tau.ac.il}

\title[Towards an optimal hypergraph container lemma]{Towards an optimal hypergraph container lemma}

\thanks{This research was supported by: the Israel Science Foundation grant 2110/22; the grant 2019679 from the United States--Israel Binational Science Foundation (BSF) and the United States National Science Foundation (NSF); and the ERC Consolidator Grant 101044123 (RandomHypGra).}

\begin{document}

\maketitle

\begin{abstract}
  The hypergraph container lemma is a powerful tool in probabilistic combinatorics that has found many applications since it was first proved a decade ago.  Roughly speaking, it asserts that the family of independent sets of every uniform hypergraph can be covered by a small number of almost-independent sets, called containers.  In this article, we formulate and prove two new versions of the lemma that display the following three attractive features.  First, they both admit short and simple proofs that have surprising connections to other well-studied topics in probabilistic combinatorics.  Second, they use alternative notions of almost-independence in order to describe the containers.  Third, they yield improved dependence of the number of containers on the uniformity of the hypergraph, hitting a natural barrier for second-moment-type approaches.
\end{abstract}

\section{Introduction}
\label{sec:introduction}

The method of hypergraph containers is a powerful and widely-applicable technique in probabilistic combinatorics.  The method enables one to control the independent sets of many `interesting' uniform hypergraphs by exploiting the fact that these sets exhibit a certain subtle clustering phenomenon.  The survey \cite{BalMorSam18} provides a gentle introduction to the method, illustrated with several example applications.

The heart of the method is a hypergraph container lemma (HCL for short) -- a statement that formalises and quantifies the notion of clustering we alluded to above.  A~generic HCL asserts that every uniform hypergraph $\cH$ admits a collection $\cC = \cC(\cH)$ of \emph{containers} for the family $\cI(\cH)$ of independent sets of $\cH$ with the following three properties:
\begin{enumerate}[label=(\roman*)]
\item
  each $I \in \cI(\cH)$ is contained in some $C \in \cC$;
\item
  \label{item:generic-HCL-few-containers}
  the family $\cC$ has `small' cardinality; and
\item
  \label{item:generic-HCL-tight-containers}
  each $C \in \cC$ is `almost independent'.
\end{enumerate}

Note that such a statement is vacuously true when $\cH$ is $1$-uniform, as then each $I \in \cI(\cH)$ is contained in the largest independent set $\{v \in V(\cH) : \{v\} \notin \cH\}$.
Although a HCL for $2$-uniform hypergraphs is already implicit in the works of Kleitman and Winston~\cite{KleWin82} from the early 1980s, an explicit statement was formulated and proved only a decade later by Sapozhenko~\cite{Sap01}.
General HCLs that are applicable to hypergraphs of an arbitrary uniformity were proved much later, independently and simultaneously, by Balogh, Morris, and Samotij~\cite{BalMorSam15} and by Saxton and Thomason~\cite{SaxTho15}.

The precise meaning of small in~\ref{item:generic-HCL-few-containers} above depends not only on the notion of almost-independence in~\ref{item:generic-HCL-tight-containers}, but also on the uniformity of the hypergraph.
The dependence on the uniformity in the original HCLs of \cite{BalMorSam15,SaxTho15} was rather unfavourable.
Roughly speaking, the guaranteed upper bound on $\log |\cC(\cH)|$ for an $r$-uniform hypergraph $\cH$ in \cite{BalMorSam15} and \cite{SaxTho15} was proportional to $e^{\Theta(r^2)}$ and $r!$, respectively.
To remedy this, Balogh and Samotij~\cite{BalSam20} proved a more `efficient' HCL that yielded an upper bound on $\log |\cC(\cH)|$ that was proportional to merely a polynomial in $r$.
Let us also mention that Morris, Samotij, and Saxton~\cite{MorSamSax24} proved an `asymmetric' HCL that was fine-tuned to enumeration of sparse graphs not containing an induced copy of a given subgraph, but has since found additional applications~\cite{Cam20,CamColMorMorSou22,CamCouSerWot23,LiuMatSza}.

While the proofs of the four HCLs mentioned in the previous paragraph are elementary, they all are rather lengthy and technical (there are much easier arguments that prove the existence of containers in the $2$-uniform case, see the survey~\cite{Sam15}).
This motivated several groups of authors to seek HCLs that admit shorter and simpler proofs.
Saxton and Thomason~\cite{SaxTho16simple} found an easy method for building containers for independent sets in simple hypergraphs.
In another work~\cite{SaxTho16online}, the same authors presented a streamlined proof of their original HCL~\cite{SaxTho15} that yields an additional `online' property.
A few years later, Bernshteyn, Delcourt, Towsner, and Tserunyan~\cite{BerDelTowTse19} formulated a HCL whose statement adapts notions from nonstandard analysis and supplied a~compact (spanning approximately four pages), non-algorithmic proof;
unfortunately, their HCL still suffers from unfavourable dependence on the uniformity.
Finally, Nenadov~\cite{Nen} recently proved a statement that may be interpreted as a probabilistic HCL; while it does not supply a collection of containers for independent sets, it is sufficiently powerful to replace HCL in most of its typical applications.

\subsection{Motivation}
\label{sec:motivation}

We embarked on this project with two independent goals in mind.  Our first aim was to further improve the dependence of (the logarithm of) the number of containers on the uniformity of the hypergraph.  The second aim was to find a simple(r) construction of containers for general uniform hypergraphs that admits a compact proof.  Whereas the first goal has been achieved only partially -- each HCL formulated and proved in this paper encounters the same `second-moment barrier' that results in $O(r^2)$-type dependence of $\log |\cC|$ on the uniformity $r$, we do suggest a pathway to improving it further.

\subsection{Our results}
\label{sec:our-results}

We formulate two new HCLs that are stronger than the efficient HCL of Balogh and the second author and, at the same time, admit short proofs.
The two lemmas use different notions of almost-independence that, unlike all previous works, are not expressed in terms of a `balanced supersaturation' condition;
however, both of them are at least as strong as the usual notion
(to certify this, we will present short derivations of the efficient HCL from each of our two lemmas).
Furthermore, one of our HCLs does not even require the input hypergraph to be uniform.

The statement of our first HCL, \Cref{HCL:cover} below, is inspired by the recent developments in the study of threshold phenomena in random sets~\cite{FraKahNarPar21,ParPha24}.
In order to state it, we need to introduce some notation.
First, given a hypergraph $\cG$, we write
\[
  \< \cG \> \coloneqq \bigcup_{E \in \cG} \{ F \subseteq V(\cG) : F \supseteq E \}
\]
for the up-set generated by $\cG$.  We say that $\cG$ \emph{covers} a hypergraph $\cH$ if $\cH \subseteq \< \cG \>$;  in other words, $\cG$ covers $\cH$ is every edge of $\cH$ contains some edge of $\cG$.
Second, for each $p \in [0,1]$, define the \emph{$p$-weight} of a hypergraph $\cG$ to be
\[
  w_p(\cG) \coloneqq \sum_{E\in \cG} p^{|E|},
\]
which is just the expected number of edges of $\cG$ induced by the $p$-random subset\footnote{The $p$-random subset of a finite set $X$ is the random set formed by independently retaining each element of $X$ with probability $p$.} of $V(\cG)$.

\begin{HCL}
  \label{HCL:cover}
  Let $\cH$ be an $r$-uniform hypergraph with a finite vertex set $V$.
  For every $p \in (0,1/(8r^2)]$, there exists a family $\cS \subseteq 2^{V}$ and functions
  \[
    g \colon \cI(\cH)\to \cS
    \qquad \text{and} \qquad
    f \colon \cS\to 2^{V}
  \]
  such that:
  \begin{enumerate}[label=(\alph*)]
  \item
    \label{item:contained}
    For each $I\in \cI(\cH)$, we have $g(I) \subseteq I \subseteq f(g(I))$.
  \item
    \label{item:small-fingerprint}
    Each $S \in \cS$ has at most $8r^2 p |V|$ elements.
  \item
    \label{item:small-cover}
    For every $S \in \cS$, letting $C \coloneqq f(S)$, there exists a hypergraph $\cG$ on $C$ with
    \[
      w_p(\cG) \le p|C|
    \]
    that covers $\cH[C]$ and satisfies $|E| \geq 2$ for all $E \in \cG$.
  \end{enumerate}
\end{HCL}

Let us point out that the above theorem does not assume anything about the input hypergraph, except that it is uniform.
The condition that all edges of the cover have size at least two is necessary to prevent the conclusion from being trivial; indeed, every hypergraph with vertex set $C$ admits a cover (by singletons) of $p$-weight $p|C|$.
An analogous comment applies to the assumption that $p \le 1/(8r^2)$;  if $p \ge 1/(8r^2)$, then one may simply set $f(I) = g(I) = I$ for all $I \in \cI(\cH)$ and let $\cG$ be the empty hypergraph.

Our second HCL, \Cref{HCL:hardcore}, has a probabilistic flavour; this stems form the fact that its proof views independent sets as samples from a hard-core model.
Perhaps surprisingly, it does not assume anything about the hypergraph, not even uniformity.
Given an arbitrary set $X$ and a real $p \in [0,1]$, we write $X_p$ to denote the $p$-random subset of $X$.

\begin{HCL}
  \label{HCL:hardcore}
  Let $\cH$ be a hypergraph with a finite vertex set $V$.
  For all reals $\delta$ and $p$ satisfying $0 < p \le \delta < 1$, there exists a family $\cS \subseteq 2^{V}$ and functions
  \[
    g \colon \cI(\cH)\to \cS
    \qquad \text{and} \qquad
    f \colon \cS\to 2^{V}
  \]
  such that:
  \begin{enumerate}[label=(\alph*)]
  \item
    \label{item:contained-hardcore}
    For each $I\in \cI(\cH)$, we have $g(I) \subseteq I \subseteq f(g(I))$.
  \item
    \label{item:small-fingerprint-hardcore}
    Each $S \in \cS$ has at most $p|V|/\delta$ elements.
  \item
    \label{item:subexponential-probability}
    For every $S \in \cS$, letting $C \coloneqq f(S)$,
    \[
      \Pr\bigl(S \cup C_p \in \cI(\cH)\bigr) \ge (1-p)^{\delta|C \setminus S|}.
    \]
  \end{enumerate}  
\end{HCL}

The proofs of both \Cref{HCL:cover,HCL:hardcore} follow the same general strategy that was also used to prove most HCLs to date; it can be traced back to the work of Kleitman and Winston~\cite{KleWin82}.
Namely, we define an algorithm that, given an independent set $I$ as input, builds sets $S$ and $C$ satisfying $S \subseteq I \subseteq C$ and, in the proof of \Cref{HCL:cover}, also a hypergraph $\cG$ that covers $\cH[C]$.
Crucially, the set $C$ depends on $I$ only through $S$ in the following sense:
If for some two inputs $I, I'$, the algorithm outputs the same set $S$, it also returns the same set $C$.
This property allows one to define the functions $g$ and $f$ via the algorithm.
The main novelty in the algorithm that underlies the proof of \Cref{HCL:cover} is that it considers queries of the form `Is $L \subseteq I$?' for sets $L$ that are not necessarily singletons.
The algorithm used in the proof of \Cref{HCL:hardcore} considers only queries of the form `Does $v \in I$?', but it chooses the vertex $v$ based on its occupancy probability in a certain hard-core model.

\subsection{Comparing \Cref{HCL:cover,HCL:hardcore}}
\label{sec:comparison}

Even though the descriptions of the container sets $f(S)$ in \Cref{HCL:cover,HCL:hardcore} look rather different, they are in fact essentially equivalent.  This is a consequence of the following proposition, which quantifies the relation between the smallest $p$-weight of a cover for a uniform hypergraph~$\cH$ and the probability that the $p$-random subset of vertices of $\cH$ is independent.

\begin{prop}
  \label{prop:prob-vs-covers}
  Suppose that $\cH$ is a hypergraph on a finite vertex set $V$.
  \begin{enumerate}[label=(\roman*)]
  \item
    \label{item:prob-vs-covers-1}
    For every hypergraph $\cG \subseteq 2^V \setminus \{\emptyset\}$ that covers $\cH$ and all $p \in (0,1/2)$,
    \[
      \Pr(V_p \in \cI(\cH)) \ge \exp\bigl(-2w_p(\cG)\bigr).
    \]
  \item
    \label{item:prob-vs-covers-2}
    If $\cH$ is $r$-uniform, then, for every $p \in (0, 1/(4r))$, there exists a hypergraph $\cG \subseteq 2^V \setminus \{\emptyset\}$ that covers $\cH$ and satisfies
    \[
      \Pr(V_p \in \cI(\cH)) \le \exp \bigl(-w_{p/(4r^2)}(\cG)/8\bigr).
    \]
  \end{enumerate}
\end{prop}

We postpone the proof of \Cref{prop:prob-vs-covers} to \Cref{sec:preliminaries} and now only elaborate on the qualitative equivalence between the notions of almost-independence used in \Cref{HCL:cover,HCL:hardcore}.
First, let $\cG$ be the hypergraph from assertion~\ref{item:small-cover} of \Cref{HCL:cover}.
The assumption that $|E| \ge 2$ for all $E \in \cG$ and the fact $w_p(\cG)\leq p|C|$ guarantee that, for all $\delta \in (0,1)$ and all $p' \le \delta p$,
\[
w_{p'}(\cG) \le (p'/p)^2 \cdot w_p(\cG) \le (p'/p) \cdot p'|C|\le \delta p'|C|.
\] 
Since $\cG$ covers $\cH[C]$, it follows from part~\ref{item:prob-vs-covers-1} of \Cref{prop:prob-vs-covers} that $\Pr\bigl(V_{p'} \in \cI(\cH[C])\bigr) \ge \exp\bigl(-2\delta p'|C|\bigr)$.

Conversely, assume that $p \le 1/(4r)$ and let $S$ and $C$ be the sets from assertion~\ref{item:subexponential-probability} of \Cref{HCL:hardcore}.
Since $\Pr\bigl(C_p \in \cI(\cH[C])\bigr) \ge \Pr\bigl(S \cup C_p \in \cI(\cH)\bigr)$, it follows from part~\ref{item:prob-vs-covers-2} of \Cref{prop:prob-vs-covers} that $\cH[C]$ admits a cover~$\cG$ that satisfies 
\[
w_{p/(4r^2)}(\cG) \le -8\log \Pr\bigl(S \cup C_p \in \cI(\cH)\bigr)\le -8\log (1-p)^{\delta |C|} \le 10 \delta p|C|\, .
\]

\subsection{Relation to earlier HCLs}
\label{sec:relation-earlier-HCLs}

An attractive feature of \Cref{HCL:cover,HCL:hardcore} is that each of them implies (a slight strengthening of) the `efficient' HCL of Balogh and Samotij~\cite[Theorem~1.1]{BalSam20}, stated here as \Cref{HCL:efficient} below, which in turn significantly improves the original HCLs proved in~\cite{BalMorSam15,SaxTho15}.
Given a hypergraph $\cH$ with vertex set $V$ and a set $L \subseteq V$, we define
\[
  \deg_{\cH}L \coloneqq |\{E \in \cH : L \subseteq E\}|.
\]
Further, for an integer $\ell \ge 1$, we let
\[
  \Delta_\ell(\cH) \coloneqq \max\{\deg_{\cH}L : L \subseteq V, |L| = \ell\}.
\]
The following statement can be easily derived from both \Cref{HCL:cover} as well as \Cref{HCL:hardcore} and Janson's inequality.  We present the two derivations in \Cref{sec:deducing-normal-containers}.

\begin{HCL}[\cite{BalSam20}]
  \label{HCL:efficient}
  Let $\cH$ be an $r$-uniform hypergraph with a finite vertex set $V$.
  Suppose that $\tau \in (0,1)$ and $K \ge r$ are such that, for every $\ell \in \br{r}$,
  \begin{equation}
    \label{HCL:efficient:Delta:bounds}
    \Delta_\ell(\cH) \leq K \cdot \left( \frac{\tau}{32 K r^2} \right)^{\ell -1} \cdot \frac{e(\cH)}{|V|}.
  \end{equation}
  There exists a family $\cS \subseteq 2^V$ and functions
  \[
    g \colon \cI(\cH)\to \cS
    \qquad \text{and} \qquad
    f \colon \cS\to 2^{V}
  \]
  such that:
  \begin{enumerate}[label=(\alph*)]
  \item
    \label{item:contained-efficient}
    For each $I\in \cI(\cH)$, we have $g(I) \subseteq I \subseteq f(g(I))$.
  \item
    \label{item:small-fingerprint-efficient}
    Each $S \in \cS$ has at most $\tau|V|$ elements.
  \item
    \label{item:small-cover-efficient}
    For every $S \in \cS$, we have $|f(S)| \le \bigl(1-1/(2K)\bigr)|V|$.
  \end{enumerate}
\end{HCL}

In most applications of the hypergraph container method, one recursively applies a basic HCL, such as \Cref{HCL:efficient}, to construct a family of containers that are almost independent in the sense that one can no longer prove a balanced supersaturation result in any of the containers.
In order to avoid explicitly performing such a recursive procedure, one may instead invoke one of the several `packaged' HCLs that exist in the literature (see, e.g., \cite[Theorem~2.2]{BalMorSam15}, \cite[Theorem~1.6]{BalSam20}, or \cite[Corollary~3.6]{SaxTho15}).
Both \Cref{HCL:cover,HCL:hardcore} allow one to deduce such a packaged HCL directly, without the need for recursion.
For example, the following statement is a straightforward consequence of \Cref{HCL:cover}.

\begin{HCL}
  \label{HCL:packaged}
  Let $\cH$ be an $r$-uniform hypergraph with a finite vertex set $V$.
  For every $p \in (0, 1/(8r^2)]$, there exists a family $\cS \subseteq 2^V$ and functions
  \[
    g \colon \cI(\cH)\to \cS
    \qquad \text{and} \qquad
    f \colon \cS\to 2^{V}
  \]
  such that:
  \begin{enumerate}[label=(\alph*)]
  \item
    \label{item:contained-packaged}
    For each $I\in \cI(\cH)$, we have $g(I) \subseteq I \subseteq f(g(I))$.
  \item
    \label{item:small-fingerprint-packaged}
    Each $S \in \cS$ has at most $8r^2p|V|$ elements.
  \item
    \label{item:non-supersaturated-packaged}
    For every $S \in \cS$, there does not exist a hypergraph $\cH_S \subseteq \cH[f(S)]$ that satisfies
    \begin{equation}
      \label{eq:p-supersaturated-def}
      \Delta_\ell(\cH_S) < \frac{p^{\ell-1}e(\cH_S)}{|f(S)|}
      \qquad \text{for all $\ell \in \{2, \dotsc, r\}$.}
    \end{equation}    
  \end{enumerate}
\end{HCL}

We note that condition~\ref{item:non-supersaturated-packaged} in the above theorem is the negation of a balanced supersaturation statement that would enable one to construct a small family of containers for independent sets in $\cH[f(S)]$ using a basic HCL such as \Cref{HCL:efficient}.

We present the derivation of \Cref{HCL:packaged} from \Cref{HCL:cover} in \Cref{sec:deducing-normal-containers} and only mention here that the existence of a hypergraph~$\cG$ with small $p$-weight that covers $\cH[f(S)]$ precludes the existence of a hypergraph $\cH_S \subseteq \cH[f(S)]$ that satisfies the balancedness condition \eqref{eq:p-supersaturated-def}.  We further note that the derivation of \Cref{HCL:packaged} from \Cref{HCL:cover} in fact allows one to prove a strengthening of the former, where the assertion \ref{item:non-supersaturated-packaged} is replaced by the stronger assertion that there does not exist a probability distribution on $\cH[f(S)]$ under which the random edge $\bR \in \cH[f(S)]$ satisfies
\[
  \Pr(L \subseteq \bR) < \frac{p^{\ell-1}}{|f(S)|} \qquad \text{for all $L \subseteq V$ with $2 \le |L| \le r$}.
\]
Note the striking similarity between this condition and the notion of a $p$-spread measure, introduced by Talagrand~\cite{Tal10}, which played a key role in the proof~\cite{FraKahNarPar21} of the fractional version of the `expectation-threshold' conjecture of Kahn and Kalai~\cite{KahKal07}, conjectured by Talagrand~\cite{Tal10}.

Finally, we remark that one may use \Cref{HCL:hardcore} to derive an alternative version of \Cref{HCL:packaged}, where assumption~\eqref{eq:p-supersaturated-def} is replaced by an analogous sequence of upper bounds on the degrees $\Delta_1(\cH_S), \dotsc, \Delta_r(\cH_S)$.  The existence of an $\cH_S \subseteq \cH[f(S)]$ that satisfies such a modified assumption would imply, via Janson's inequality, an upper bound on the probability that the $p$-random set $f(S)_p$ is independent in $\cH_S$, and thus also in $ \subseteq \cH[f(S)]$, that is smaller than the upper bound on this probability asserted by \Cref{HCL:hardcore}.

\subsection{Interpolating between \Cref{HCL:cover,HCL:hardcore}}
\label{sec:interpolating}

Before concluding our discussion, we state one more HCL, where the characterisation of almost independence interpolates between those given by \Cref{HCL:cover,HCL:hardcore}.

\begin{HCL}
  \label{HCL:hardcore-cover}
  Let $\cH$ be a hypergraph with a finite vertex set $V$.
  For all reals $\delta$ and $p$ satisfying $0 < p \le \delta < 1$, there exists a family $\cS \subseteq 2^{V}$ and functions
  \[
    g \colon \cI(\cH)\to \cS
    \qquad \text{and} \qquad
    f \colon \cS\to 2^{V}
  \]
  such that:
  \begin{enumerate}[label=(\alph*)]
  \item
    \label{item:contained-hardcore-cover}
    For each $I\in \cI(\cH)$, we have $g(I) \subseteq I \subseteq f(g(I))$.
  \item
    \label{item:small-fingerprint-hardcore-cover}
    Each $S \in \cS$ has at most $p|V|/\delta$ elements.
  \item
    \label{item:description-containers-hardcore-cover}
    For every $S \in \cS$, letting $C \coloneqq f(S)$, there exists a hypergraph $\cG$ with vertex set $C \setminus S$ and $|E| \ge 2$ for all $E \in \cG$ that covers $\cH[C]$ and satisfies
    \[
      \Pr\bigl(L \subseteq C_p \mid C_p \in \cI(\cG)\bigr) \ge (1-\delta)^{|L|} p^{|L|}
    \]
    for all $L \in \cI(\cG)$.
  \end{enumerate}
\end{HCL}

The proof of \Cref{HCL:hardcore-cover}, which we sketch in \Cref{sec:interpolating-cover-hardcore}, is very similar to the proof of \Cref{HCL:hardcore}, the only difference being that the algorithm that is used here to build the sets $S$ and $C$ considers the more general queries of the form `Is $L \subseteq I$?', as in the proof of \Cref{HCL:cover}.

While it may not be immediately clear, \Cref{HCL:hardcore-cover} simultaneously strengthens both \Cref{HCL:cover,HCL:hardcore}.  We postpone the proofs of these two facts (\Cref{HCL:hardcore-cover}$\implies$\Cref{HCL:cover} and \Cref{HCL:hardcore-cover}$\implies$\Cref{HCL:hardcore}) to \Cref{sec:interpolating-cover-hardcore}.

\subsection{Organisation}

In \Cref{sec:preliminaries}, we set a few notational conventions, state several auxiliary results, and prove \Cref{prop:prob-vs-covers}.  In \Cref{sec:proof-HCL-cover,sec:proof-HCL-hardcore},  we prove \Cref{HCL:cover,HCL:hardcore}, respectively.  \Cref{sec:deducing-normal-containers} is devoted to derivations of standard HCLs (\Cref{HCL:efficient,HCL:packaged}) from \Cref{HCL:cover,HCL:hardcore}.  In \Cref{sec:interpolating-cover-hardcore}, we sketch the proof of \Cref{HCL:hardcore-cover} and show that it implies both \Cref{HCL:cover,HCL:hardcore}.  In \Cref{sec:concluding-remarks}, we present an attractive conjecture that strengthens \Cref{prop:prob-vs-covers} and discuss some of its implications.  Finally, \Cref{sec:three-proofs-key-prop} contains three proofs of one of our key auxiliary results, \Cref{prop:key}.



\section{Preliminaries}
\label{sec:preliminaries}

\subsection{Notation}
\label{sec:notation}

Given a hypergraph $\cH$ and a positive integer $s$, we define
\[
  \cH^s \coloneqq \{E \in \cH : |E| = s\}.
\]
It will be also convenient to introduce the shorthand notations
\[
  \cH^{>1} \coloneqq \bigcup_{s > 1} \cH^s
  \qquad
  \text{and}
  \qquad
  \cH^{<r} \coloneqq \bigcup_{s < r} \cH^s.
\]
Finally, for every $L \subseteq V(\cH)$, we denote by $\partial_L\cH$ the link of $L$ in $\cH$, i.e.,
\[
  \partial_L\cH \coloneqq \{E \setminus L : L \subseteq E \in \cH\};
\]
for the sake of brevity, we will write $\partial_v \cH$ in place of $\partial_{\{v\}}\cH$.

\subsection{Auxiliary results}
\label{sec:auxiliary}

While comparing \Cref{HCL:hardcore} to other HCLs, we will crucially use the following well-known inequality due to Janson~\cite{Jan90}.

\begin{theorem}[Janson's Inequality]
  \label{thm:Janson}
  Suppose that $\cG$ is a hypergraph on a finite set~$C$.
  For every $p \in [0,1]$,
  \[
    \Pr(C_p \in \cI(\cG)) \le \exp\left(- \frac{\mu^2}{2\Delta^*}\right),
  \]
  where
  \[
    \mu \coloneqq \sum_{A \in \cG} p^{|A|}
    \qquad
    \text{and}
    \qquad
    \Delta^* \coloneqq \sum_{\substack{A, B \in \cG \\ A \cap B \neq \emptyset}} p^{|A \cup B|}
  \]
  and the second sum ranges over all ordered pairs ($A$ and $B$ are not necessarily distinct).
\end{theorem}

Our proof of \Cref{HCL:hardcore}, as well as the derivation of this theorem from \Cref{HCL:hardcore-cover},  requires the following probabilistic estimate.  Since this estimate follows from standard techniques, we postpone its proof to the appendix (where we in fact present three different proofs).

\begin{prop}
  \label{prop:key}
  Suppose that $C$ is a finite set and let $\cI \subseteq 2^C$ be a decreasing family of subsets. For every $p \in (0,1)$, we have
  \[
    \log \Pr(C_p \in \cI) \ge \left(|C| - \frac{\Ex\bigl[|C_p| \mid C_p \in \cI\bigr]}{p}\right) \cdot \log(1-p).
  \]
\end{prop}

Finally, we will also employ the following well-known inequality, discovered independently by Bollobás~\cite{Bol65}, Lubell~\cite{Lub66}, Meshalkin~\cite{Mes63}, and Yamamoto~\cite{Yam54}.

\begin{prop}[LYMB inequality]
  \label{prop:LYMB}
  Suppose that $X$ is a finite set and that $\cA \subseteq 2^X$ is an antichain.  Then,
  \[
    \sum_{A \in \cA} \binom{|X|}{|A|}^{-1} \le 1.
  \]  
\end{prop}

\subsection{Proof of \Cref{prop:prob-vs-covers}}
\label{sec:proof-prob-vs-covers}

The first assertion is a straightforward application of Harris's inequality.
Indeed, for every hypergraph $\cG \subseteq 2^V \setminus \{\emptyset\}$ that covers $\cH$, we have
\[
  \Pr(V_p \in \cI(\cH)) \ge \Pr(V_p \in \cI(\cG)) \ge \prod_{A \in \cG} (1-p^{|A|}) \ge \exp\left(-2\sum_{A \in \cG}p^{|A|}\right),
\]
where the final inequality holds as $1-x \ge e^{-2x}$ for all $x < 1/2$.

We now prove the second assertion.
Suppose that $\cH$ is $r$-uniform and define, for each $\ell \in \br{r}$, $\lambda_\ell \coloneqq 4^{-\min\{\ell, r-\ell\}}$.
Let $\cH'$ be a maximal subgraph of $\cH$ subject to $\Delta_\ell(\cH') \le \lambda_\ell \binom{r}{\ell}^{-1}p^{\ell-r}$ for all $\ell \in \br{r}$.
Define $\mu \coloneqq \sum_{A \in \cH'} p^{|A|} = e(\cH') p^r$ and
\[
  \Delta^* \coloneqq \sum_{A \in \cH'} p^{|A|} \cdot \sum_{\substack{B \in \cH' \\ A \cap B \neq \emptyset}} p^{|B\setminus A|}
  \le \mu \cdot \sum_{\ell=1}^r \binom{r}{\ell}\Delta_\ell(\cH')p^{r-\ell} \le \mu \cdot \sum_{\ell=1}^r \lambda_\ell \le 2\mu.
\]
Let $\cG$ be the set of minimal elements in the family
\[
  \cG' \coloneqq
  \left\{
    T \subseteq V : 1 \le |T| \le r \wedge \deg_{\cH'}T = \left\lfloor \lambda_{|T|}\binom{r}{|T|}^{-1}p^{|T|-r} \right\rfloor
  \right\}
\]
and note that $\cH \subseteq \<\cG'\> = \<\cG\>$ by maximality of $\cH'$.
Since $\cG$ is an antichain, the LYMB inequality (\Cref{prop:LYMB}) yields
\[
  e(\cH') \ge \sum_{A \in \cH'} \sum_{\substack{T \in \cG \\ T \subseteq A}} \binom{r}{|T|}^{-1} = \sum_{T \in \cG} \binom{r}{|T|}^{-1} \deg_{\cH'}T.
\]
Since our upper-bound assumption on $p$ yields
\[
  \lambda_\ell\binom{r}{\ell}^{-1}p^{\ell-r} \ge (4 r p)^{\ell-r} \ge 1
\]
for every $\ell \in \br{r}$, we may conclude that
\[
  \mu = e(\cH')p^r \ge \frac{1}{2} \sum_{T \in \cG} \lambda_{|T|}\binom{r}{|T|}^{-2}p^{|T|} \ge \frac{1}{2}\sum_{T \in \cG} \left(\frac{p}{4r^2}\right)^{|T|} = \frac{w_{p/(4r^2)}(\cG)}{2}.
\]
We may now apply Janson's inequality (\Cref{thm:Janson}) to conclude that
\[
  \Pr(V_p \in \cI(\cH)) \le \Pr(V_p \in \cI(\cH')) \le \exp\left(-\frac{\mu^2}{2\Delta^*}\right) \le \exp\left(-\frac{\mu}{4}\right) \le \exp\left(-\frac{w_{p/(4r^2)}(\cG)}{8}\right),
\]
as desired.

\section{Proof of \Cref{HCL:cover}}
\label{sec:proof-HCL-cover}


\subsection{The algorithm}

Assume that $\cH$ is an $r$-uniform hypergraph on a finite set $V$ and fix some $p \in (0,1/(8r^2)]$.
The following algorithm takes as input an independent set $I \in \cI(\cH)$ and returns sets $S, C \subseteq V$ and a hypergraph $\cG$ on $V$.

\begin{enumerate}[label=(\arabic*)]
\item
  Let $\cH_0 \coloneqq \cH$ and $S_0 \coloneqq \emptyset$.
\item
  Do the following for $i = 0, 1, \dotsc$:
  \begin{enumerate}[label=(\alph*),ref=(\arabic{enumi}\alph*)]
  \item
    Let $C_i \coloneqq \{v \in V : \{v\} \notin \cH_i\}$.
  \item
    \label{item:stopping-condition}
    If $w_p(\cH_i^{>1}) \le p|C_i|$, then set $J \coloneqq i$ and \texttt{STOP}.
  \item
    \label{item:dfn-si-Li}
    Otherwise, let $s_i$ be the smallest $s \in \{2, \dotsc, r\}$ such that $w_p(\partial_L \cH_i^s) \ge 1/(4r)$ for some nonempty $L \in 2^V \setminus \cH_i$ and let $L_i$ be an inclusion-maximal set $L$ with this property.\footnote{It may not be clear at this point that such $s$ and $L$ exist, but we will prove this shortly.}
  \item
    \label{item:Li-in-I}
    If $L_i \subseteq I$, then set $S_{i+1} \coloneqq S_i \cup L_i$ and $\cF_i \coloneqq \partial_{L_i} \cH_i^{s_i}$.
  \item
    \label{item:Li-not-in-I}
    Otherwise, if $L_i \nsubseteq I$, set $S_{i+1} \coloneqq S_i$ and $\cF_i \coloneqq \{L_i\}$.
  \item
    Finally, set $\cH_{i+1} \coloneqq (\cH_i \setminus \<\cF_i\>) \cup \cF_i$.
  \end{enumerate}
\item
  Return $S \coloneqq S_J$, $C \coloneqq C_J$, and $\cG \coloneqq \cH_J^{>1}$.
\end{enumerate}

\subsection{Well-definedness}
\label{sec:well-definedness-cover}

We now show that the algorithm described in the previous section is well defined and that it terminates on every input.
Our first lemma takes care of the former and shows that the definition of $s_i$ and $L_i$ in step \ref{item:dfn-si-Li} is always valid.

\begin{lemma}\label{lem:largelink}
  If $w_p(\cH_i^{>1}) \geq p |C_i|$, then there are $s \in \{2, \dotsc, r\}$ and nonempty $L \in 2^V \setminus \cH_i$ with $w_p(\partial_L\cH_i^s) \ge 1/r$.
\end{lemma}
\begin{proof}
  Note first that our assumption implies that
  \[
    p|C_i| \le w_p(\cH_i^{>1}) = \sum_{s = 2}^r w_p(\cH_i^s) \le \sum_{s=2}^r s \cdot w_p(\cH_i^s) = \sum_{s=2}^r \sum_{v \in C_i} p \cdot w_p(\partial_v\cH_i^s).
  \]
  Consequently, there must be some $v \in C_i$ such that
  \[
    \sum_{s=2}^r w_p(\partial_v \cH_i^s) \ge 1.
  \]
  Since the fact that $v \in C_i$ guarantees that $\{v\} \notin \cH_i$, we may take $L \coloneqq \{v\}$ and an index~$s$ that achieves $\max_s w_p(\partial_v\cH_i^s) \ge 1/r$.
\end{proof}

Next, we show that each hypergraph $\cH_i$ defined in the algorithm is an antichain and that the sequence $(\<\cH_i\>)_i$ is strictly increasing.  Note that this implies that the algorithm terminates on every input.  Indeed, the fact that $\cH_i$ is an antichain means that it comprises precisely all the minimal elements of $\<\cH_i\>$ and therefore $\cH_0, \cH_1, \dotsc$ are pairwise distinct.

\begin{lemma}
  \label{lemma:Hi-antichain}
  The hypergraph $\cH_i$ is an antichain for every $i$.
\end{lemma}
\begin{proof}
  We prove the assertion of the lemma by induction on $i$.  The induction basis holds due to our assumption that $\cH_0 = \cH$ is $r$-uniform.  For the induction step, assume that $\cH_i$ is an antichain for some $i \ge 0$.  Observe first that $\cF_i$ is an antichain as well.  Indeed, this is clear when $\cF_i = \{L_i\}$; otherwise $\cF_i \subseteq \partial_{L_i} \cH_i$ and this is a straightforward consequence of the assumption that $\cH_i$ is an antichain.
  Suppose that $\cH_{i+1} = (\cH_i \setminus \<\cF_i\>) \cup \cF_i$ contained a pair of sets $E, F$ with $E \subsetneq F$.  Since both $\cH_i$ and $\cF_i$ are antichains, then either $E \in \cF_i$ and $F \in \cH_i \setminus \<\cF_i\>$, which is clearly impossible, or vice-versa.  However, if $F \in \cF_i$, then $F \subseteq G$ for some $G \in \cH_i$.  (Indeed, when $\cF_i = \partial_{L_i} \cH_i^{s_i}$, we may take $G = F \cup L_i$; otherwise, $F = L_i$ and the existence of such a $G$ follows from the fact that $w_p(\partial_{L_i}\cH_i) > 0$.)  As a result, if there was $E \in \cH_i$ with $E \subsetneq F$, then $E \subsetneq G$ for some $G \in \cH_i$, contradicting the inductive assumption that $\cH_i$ is an antichain.
\end{proof}

\begin{lemma}
  \label{lemma:upset-strictly-increasing}
  We have $\<\cH_i\> \subsetneq \<\cH_{i+1}\>$ and $\cF_i \cap \<\cH_i\> = \emptyset$ for every $i$.
\end{lemma}
\begin{proof}
  The definition of $\cH_{i+1}$ implies that $\<\cH_{i+1}\> = \<\cH_i\> \cup \<\cF_i\> \supseteq \<\cH_i\>$.  To complete the proof of the lemma, it is therefore enough to show that $\cF_i \neq \emptyset$ and $\cF_i \cap \<\cH_i\> = \emptyset$.  The former statement is clear when $\cF_i = \{L_i\}$; otherwise, it follows from the assumption that $\partial_{L_i} \cH_i^{s_i} \neq \emptyset$.  For the latter statement, since $\cH_i$ is an antichain (by \Cref{lemma:Hi-antichain}), it is enough to argue that every element of $\cF_i$ is a proper subset of some element of $\cH_i$.  When $\cF_i = \{L_i\}$, this is the case since $L_i \notin \cH_i$ and $\partial_{L_i} \cH_i \neq \emptyset$; otherwise, if $\cF_i = \partial_{L_i} \cH_i^{s_i}$, this follows as $L_i \neq \emptyset$.
\end{proof}

\subsection{Properties}
\label{sec:properties-cover}

We now turn to establishing some key properties of the algorithm.  We will show that, for every input $I \in \cI(\cH)$, we have $S \subseteq I \subseteq C$, the set $S$ has at most $8r^2p|V|$ elements, and the hypergraph $\cG$ (which clearly satisfies $w_p(\cG) \le p|C|$, see the stopping condition~\ref{item:stopping-condition}, and $|E| \ge 2$ for all $E \in \cG$) covers $\cH[C]$.  Further, we will prove that the set $C$ (and the hypergraph $\cG$) depends on $I$ only through $S$, which will allow us to define the functions $f$ and $g$ by $g(I) \coloneqq S$ and $f(g(I)) \coloneqq C$.

\begin{lemma}
  \label{lem:cover}
  The hypergraph $\cG$ covers $\cH[C]$.
\end{lemma}
\begin{proof}
  Note first that \Cref{lemma:upset-strictly-increasing} implies that $\cH = \cH_0 \subseteq \<\cH_0\> \subseteq \<\cH_J\>$.
  Let $F$ be an arbitrary edge of $\cH[C]$.  Since $\cH_J$ covers $\cH$, there must be some $E \in \cH_J$ with $E \subseteq F$.  However, as $F \subseteq C = C_J$, the set $E$ cannot be a singleton and thus $E \in \cH_J^{>1} = \cG$, as required.
\end{proof}

\begin{lemma}
  \label{lem:remain-ind-cover}
  For each $i$, we have $I \in \cI(\cH_i)$ and $S_i \subseteq I \subseteq C_i$.
\end{lemma}
\begin{proof}
  It suffices to show that $S_i \subseteq I \in \cI(\cH_i)$ for all $i$;  indeed, $I \in \cI(\cH_i)$ implies that $I \subseteq C_i$, as $\{v\} \in \cH_i$ for every $v \in V \setminus C_i$.  We prove this by induction on $i$.  The induction basis follows as $S_0 = \emptyset$ and $I \in \cI(\cH) = I(\cH_0)$.  For the induction step, assume that $S_i \subseteq I \in \cI(\cH_i)$ for some $i \ge 0$.  Note first that either $S_{i+1} = S_i$ or $S_{i+1} \setminus S_i \subseteq L_i \subseteq I$ and thus $S_{i+1} \subseteq I$.  Further, since $\cH_{i+1} \subseteq \cH_i \cup \cF_i$, it suffices to show that $I \in \cI(\cF_i)$.  If $L_i \nsubseteq I$, then $\cF_i = \{L_i\}$ and thus $\cI(\cF_i)$ comprises all subsets of $V$ not containing $L_i$ (which includes $I$).  Otherwise, if $L_i \subseteq I$, we have $\cF_i \subseteq \partial_{L_i}\cH_i$ and thus $\cI(\cF_i) \supseteq \cI(\partial_{L_i}\cH_i)$;  further, $\cI(\partial_{L_i}\cH_i)$ contains all sets $I' \subseteq V$ such that $L_i \cup I' \in \cI(\cH_i)$ (which again includes $I = L_i \cup I$).
\end{proof}

\begin{lemma}
  \label{lemma:cFi-uniform}
  For every $i$, the hypergraph $\cF_i$ is $u$-uniform for some $u < s_i$.
\end{lemma}
\begin{proof}
  Suppose first that $\cF_i = \{L_i\}$.  In this case, $\cF_i$ is clearly $|L_i|$-uniform and $|L_i| < s_i$ as $w_p(\partial_{L_i}\cH_i^{s_i}) > 0$ and $L_i \notin \cH_i$.  Suppose now that $\cF_i = \partial_{L_i} \cH_i^{s_i}$.  In this case, $\cF_i$ is clearly $(s_i-|L_i|)$-uniform and $s_i - |L_i| < s_i$ as $L_i$ is nonempty.
\end{proof}

\begin{lemma}
  \label{lemma:wp-partial-L-bound}
  For every $i$ and all $s$ and $L \subseteq V$ with $1 \le |L| < s < r$,
  \[
    w_p(\partial_L \cH_i^s) \le \frac{1}{2r}.
  \]
\end{lemma}
\begin{proof}
  We prove the lemma by induction on $i$.
  The base case $i = 0$ is trivial, as $\cH$ is $r$-uniform and thus the hypergraphs $\cH_0^2, \dotsc, \cH_0^{r-1}$ are all empty.
  For the induction step, assume that the assertion holds for some $i \ge 0$ (and all $s$ and $L$).
  Since $\cH_{i+1} \subseteq \cH_i \cup \cF_i$, we have
  \begin{equation}
    \label{eq:wp-partial-L-change}
    w_p(\partial_L \cH_{i+1}^s) \le w_p(\partial_L\cH_i^s) + w_p(\partial_L \cF_i^s)
  \end{equation}
  for all $s$ and all $L \subseteq V$.  Further, since the hypergraph $\cF_i$ is $u$-uniform for some $u < s_i$ (by \Cref{lemma:cFi-uniform}), we have $w_p(\partial_L\cF_i^s) = 0$ unless $s = u$.  We may therefore assume from now on that $s = u$, as otherwise the desired inequality (for all $L \subseteq V$ with $1 \le |L| < s$) follows from~\eqref{eq:wp-partial-L-change} and the inductive assumption.  The minimality of $s_i$ and the fact that $u < s_i$ imply that $w_p(\partial_L \cH_i^u) < 1/(4r)$.  It thus suffices to show that $w_p(\partial_L \cF_i) < 1/(4r)$ for every $L \subseteq V$ with $|L| < u$.  Fix some such set $L$.
  If $\cF_i = \partial_{L_i}\cH_i^{s_i}$, then the maximality of $L_i$ implies that $w_p(\partial_L\cF_i) = \1_{L \cap L_i = \emptyset} \cdot w_p(\partial_{L \cup L_i}\cH_i^{s_i}) < 1/(4r)$.  Otherwise, if $\cF = \{L_i\}$, we have
  \[
    w_p(\partial_L\cF_i) = \1_{L \subseteq L_i} \cdot p^{|L_i|-|L|} = \1_{L \subseteq L_i} \cdot p^{u-|L|} \le p \le 1/(4r),
  \]
  as desired.
\end{proof}

\begin{corollary}
  \label{cor:wp-partial-L-bound}
  For every $i$ and all $L \in 2^V \setminus \<\cH_i\>$,
  \[
    w_p(\partial_L\cH_i^{<r}) \le \frac12.
  \]
\end{corollary}

\begin{lemma}
  \label{lem:weight-increase}
  For every $i \in \{0, \dotsc, J-1\}$,
  \[
    w_p(\cH_{i+1}^{<r}) \geq w_p(\cH_i^{<r}) + \frac{1}{8r} \cdot \1_{L_i \subseteq I}.
  \]
\end{lemma}
\begin{proof}
  Since the hypergraph $\cF_i$ is $u$-uniform for some $u < s_i \le r$ (by \Cref{lemma:cFi-uniform}), it follows from the definition of $\cH_{i+1}$ that
  \[
    w_p( \cH_{i+1}^{<r} ) - w_p( \cH_i^{<r} ) \ge \sum_{L \in \cF_i} p^{|L|} \cdot \bigl(1 - w_p(\partial_{L}\cH_i^{<r})\bigr).
  \]
  Further, since $\cF_i \cap \<\cH_i\> = \emptyset$ (by \Cref{lemma:upset-strictly-increasing}), \Cref{cor:wp-partial-L-bound} implies that $w_p(\partial_L\cH_i^{<r}) \le 1/2$ for each $L \in \cF_i$ and thus $w_p( \cH_{i+1}^{<r} ) - w_p( \cH_i^{<r} ) \ge w_p(\cF_i)/2$.
  The assertion of the lemma follows as $w_p(\cF_i) \ge 0$ and $w_p(\cF_i) = w_p(\partial_{L_i}\cH_i^{s_i}) \ge 1/(4r)$ when $L_i \subseteq I$.
\end{proof}

\begin{lemma}
  \label{lem:small-fingerprint-cover}
  We have $|S|\leq 8r^2 p|V|$.
\end{lemma}
\begin{proof}
  Let $m$ denote the number of rounds of the algorithm in which $L_i \subseteq I$.
  Since $\cH_0$ is $r$-uniform, it follows from \Cref{lem:weight-increase} that
  \[
    w_p(\cH_J^{<r}) = w_p(\cH_J^{<r}) - w_p(\cH_0^{<r}) = \sum_{i=0}^{J-1} \bigl(w_p(\cH_{i+1}^{<r}) - w_p(\cH_i^{<r})\bigr) \ge \frac{m}{8r}.
  \]
  On the other hand, since $\cH_J = \cH_J^1 \cup \cG$ and $w_p(\cG) \le p|C_J|$, we have
  \[
    w_p(\cH_J) = w_p(\cH_J^1) + w_p(\cG) = |V \setminus C_J| \cdot p + w_p(\cG) \le p|V|.
  \]
  We conclude that $m \le 8r|V|$. Since $|S_{i+1}| - |S_i| \le r$ when $L_i \subseteq I$ and $|S_{i+1}| = |S_i|$ otherwise, the assertion of the lemma follows.
\end{proof}

\begin{lemma}
  \label{lem:fing-def}
  Suppose that the algorithm with inputs $I, I' \in \cI(\cH)$ outputs $(S,C,\cG)$ and $(S',C',\cG')$, respectively.  If $S=S'$, then $(C,\cG) = (C', \cG')$.
\end{lemma}

\begin{proof}
  Observe first that the execution of the algorithm on a given input (an independent set of~$\cH$) depends solely on how the set $S_{i+1}$ and the hypergraph $\cF_{i+1}$ are defined in each round $i$ (via step \ref{item:Li-in-I} or step \ref{item:Li-not-in-I}).  Thus, if the outputs of the algorithm applied to $I$ and $I'$ are different, then there must be some $i$ for which $L_i' \subseteq I'$ but $L_i \nsubseteq I$, or vice-versa.  Let $i$ be the \emph{earliest} such round and note that $\cH_i' = \cH_i$ and thus also $L'_i = L_i$.  We may clearly assume that $L_i = L'_i \subseteq I'$ and $L_i \nsubseteq I$.  However, \Cref{lem:remain-ind-cover} implies that $L_i = L'_i \subseteq S'_{i+1} \subseteq S' = S \subseteq I$, a~contradiction. 
\end{proof}

We finish our discussion with a formal proof of \Cref{HCL:cover}.

\begin{proof}[Proof of \Cref{HCL:cover}]
  For each independent set $I \in \cI(\cH)$, set $g(I) \coloneqq S$ and $f(S) \coloneqq C$, where $(S,C,\cG)$ is the output of the algorithm with input $I$.  (Recall that \Cref{lemma:Hi-antichain,lemma:upset-strictly-increasing} imply that the algorithm terminates on every input.)  Further, set $\cS \coloneqq \{ g(I) : I \in \cI(\cH) \}$.  \Cref{lem:fing-def} guarantees that the function $f$ is well-defined.  By \Cref{lem:remain-ind-cover}, we have $g(I) \subseteq I \subseteq f(g(I))$ for every $I \in \cI(\cH)$ whereas \Cref{lem:small-fingerprint-cover} shows that $|S|\leq 8r^2 pn$ for every $S \in \cS$.  Finally, writing $C = f(g(I))$, \Cref{lem:cover} guarantees that the hypergraph $\cG$ is a cover for $\cH[C]$ with $w_p(\cG) \le p|C|$; the property that $|E| \geq 2$ for all $E \in \cG$ is straightforward from the definition of $\cG$.  This concludes the proof of the theorem.
\end{proof}

\section{Proof of~\Cref{HCL:hardcore}}
\label{sec:proof-HCL-hardcore}

\subsection{The algorithm}

Assume that $\cH$ is a hypergraph on a finite set $V$ and fix some $p \in (0,1]$.
The following algorithm takes as input an independent set $I \in \cI(\cH)$ and returns sets $S, C \subseteq V$.

\begin{enumerate}[label=(\arabic*)]
\item
  Let $\cH_0 \coloneqq \cH$ and $S_0 \coloneqq \emptyset$.
\item
  Do the following for $i = 0, 1, \dotsc$:
  \begin{enumerate}[label=(\alph*),ref=(\arabic{enumi}\alph*)]
  \item
    \label{item:stopping-condition-hardcore}
    If there exists a vertex $v \in V \setminus S_i$ such that $\{v\} \notin \cH_i$ and
    \[
      \Pr\bigl(v \in V_p \mid V_p \in \cI(\cH_i)\bigr) < (1-\delta)p,
    \]
    let $v_i$ be some such vertex;  otherwise, set $J \coloneqq i$ and \texttt{STOP}.
  \item
    If $v_i \in I$, set $S_{i+1} \coloneqq S_i \cup \{v_i\}$ and $\cH_{i+1} \coloneqq \cH_i \cup \partial_{v_i}\cH_i$.
  \item
    Otherwise, if $v_i \notin I$, set $S_{i+1} \coloneqq S_i$ and $\cH_{i+1} \coloneqq \cH_i \cup \{\{v_i\}\}$.
  \end{enumerate}
\item
  Return $S \coloneqq S_J$ and $C \coloneqq \left\{v \in V : \{v\} \notin \cH_J \right\}$.
\end{enumerate}

\subsection{Properties}
\label{sec:properties-hardcore}

We now establish some key properties of the algorithm.
We first claim that the algorithm always terminates with $J \le |V|$.
Indeed, for each $i$, the vertex $v_i$ is either added to $S_i$ or $\{v_i\}$ becomes an edge of $\cH_i$; in both cases, $v_i$ is never considered again in step \ref{item:stopping-condition-hardcore}.
Further, we claim that the set $C$ can be reconstructed with the knowledge of $S$ only.
Indeed, for all $i$, the hypergraph $\cH_i$ depends on $S_i$ only.
Therefore, running the algorithm with input $S$ instead of $I$ results in the same sets $S$ and~$C$.
This allows us to define the functions $f$ and $g$ by $g(I) \coloneqq S$ and $f(g(I)) \coloneqq C$.
It remains to show that, for every input $I \in \cI(\cH)$, we have $S \subseteq I \subseteq C$, the set $S$ has at most $p|V|/\delta$ elements, and that $S \cup C_p \in \cI(\cH)$ with probability at least $(1-p)^{\delta|C|}$.

\begin{lemma}
  \label{lem:remain-ind-hardcore}
  For each $i$, we have $S_i \subseteq I \in \cI(\cH_i)$ and $I' \in \cI(\cH_i) \iff S_i \cup I' \in \cI(\cH_i)$.
\end{lemma}
\begin{proof}
  We prove the assertion of the lemma by induction on $i$.  The induction basis follows as $S_0 = \emptyset$ and $\cH = \cH_0$.  For the induction step, assume that the assertion holds for some $i \ge 0$.  Suppose first that $v_i \in I$.  In this case, $S_{i+1} = S_i \cup \{v_i\} \subseteq I$ and $\cH_{i+1} = \cH_i \cup \partial_{v_i}\cH_i$, which means that $\cI(\cH_{i+1}) = \{I' \subseteq V : \{v_i\} \cup I' \in \cI(\cH_i)\} \ni I$.  Further
  \[
    I' \in \cI(\cH_{i+1}) \iff \{v_i\} \cup I' \in \cI(\cH_i) \iff S_i \cup \{v_i\} \cup I' \in \cI(\cH_i) \iff S_{i+1} \cup I' \in \cI(\cH_{i+1}).
  \]
  Suppose now that $v_i \notin I$.  In this case, $S_{i+1} = S_i \subseteq I$ and $\cH_{i+1} = \cH_i \cup \{\{v_i\}\}$, which means that $\cI(\cH_{i+1}) = \{I' \in \cI(\cH_i) : v_i \notin I'\} \ni I$.  Further, $I' \in \cI(\cH_{i+1}) \iff S_{i+1} \cup I' \in \cI(\cH_{i+1})$, as $S_{i+1} = S_i \not\ni v_i$. 
\end{proof}

The above lemma clearly implies that $S = S_J \subseteq I$.  Further, $I \in \cI(\cH_J) \subseteq 2^C$, as no vertex $v$ satisfying $\{v\} \in \cH_J$ can belong to an independent set of $\cH_J$.  This establishes $S \subseteq I \subseteq C$.

\begin{lemma}
  \label{lem:small-fingerprint-hardcore}
  We have $\delta|S| \le p|V|$.
\end{lemma}
\begin{proof}
  We claim that the sequence of probabilities $P_i \coloneqq \Pr\bigl(V_p \in \cI(\cH_i)\bigr)$ satisfies
  \[
    P_{i+1} \le (1-\delta)^{\1_{v_i \in I}} \cdot P_i,
  \]
  for all $i \in \{0, \dotsc, J-1\}$, which gives
  \begin{equation}
    \label{eq:PJ-P0-ratio}
    (1-p)^{|V|} = \Pr(V_p = \emptyset) \le P_J \le (1-\delta)^{|\{i : v_i \in I\}|} \cdot P_0 = (1-\delta)^{|S|} \cdot P_0 \le (1-\delta)^{|S|}.
  \end{equation}
  Indeed, if $v_i \notin I$, then the inequality $P_{i+1} \le P_i$ follows as $\cH_i \subseteq \cH_{i+1}$, and thus $\cI(\cH_{i+1}) \subseteq \cI(\cH_i)$.
  On the other hand, if $v_i \in I$, then $\cH_{i+1} = \cH_i \cup \partial_{v_i} \cH_i$, and thus $\cI(\cH_{i+1}) = \{I' \subseteq V : I' \cup \{v_i\} \in \cI(\cH_i)\}$.  Consequently, by the choice of $v_i$,
  \[
    p \cdot P_{i+1} = \Pr\bigl(v_i \in V_p \in \cI(\cH_i)\bigr) = \Pr(v_i \in V_p \mid V_p \in \cI(\cH_i)) \cdot P_i < (1-\delta)p \cdot P_i,
  \]
  which gives the desired inequality.
  We claim that \eqref{eq:PJ-P0-ratio} implies the desired inequality $\delta|S| \le p|V|$.  Indeed, if $\delta|S| > p|V|$, then $(1-\delta)^{|S|} < (1-\delta)^{p|V|/\delta} \le (1-p)^{|V|}$, where the second inequality follows as $p \le \delta$ and the map $x \mapsto (1-x)^{1/x}$ is decreasing on $(0,1)$.
\end{proof}

\begin{lemma}
  \label{lemma:tight-container-hardcore}
  We have $\Pr\bigl(S \cup C_p \in \cI(\cH)\bigr) \ge (1-p)^{\delta |C \setminus S|}$.
\end{lemma}
\begin{proof}
  Define $C' \coloneqq C \setminus S$.  Since clearly $\cH = \cH_0 \subseteq \cH_1 \subseteq \dotsb \subseteq \cH_J$, we have
  \[
    \Pr(S \cup C_p \in \cI(\cH)) \ge \Pr(S \cup C_p \in \cI(\cH_J)) = \Pr(S \cup C_p' \in \cI(\cH_J)).
  \]
  In view of \Cref{prop:key}, the assertion of the lemma will follow if we prove that
  \begin{equation}
    \label{eq:ExCp'-lower}
    \Ex\bigl[|C_p'| \mid S \cup C_p' \in \cI(\cH_J)\bigr] \ge (1-\delta)p|C'|.
  \end{equation}
  To this end, fix an arbitrary $v \in C'$.  Since $\cI(\cH_J) \subseteq 2^C$, as we have already observed above, and since $S \cup V_p \in \cI(\cH_J) \iff V_p \in \cI(\cH_J)$, by \Cref{lem:remain-ind-hardcore}, we have
  \[
    \Pr\bigl(v \in C_p' \mid S \cup C_p' \in \cI(\cH_J)\bigr) = \Pr\bigl(v \in V_p \mid S \cup V_p \in \cI(\cH_J)\bigr) = \Pr\bigl(v \in V_p \mid V_p \in \cI(\cH_J)\bigr).
  \]
  Finally, since $v \in C' = C \setminus S$, the stopping condition~\ref{item:stopping-condition-hardcore} implies that the above probability is at least $(1-\delta)p$.  Summing this inequality over all $v \in C'$ yields \eqref{eq:ExCp'-lower}.
\end{proof}

We conclude our discussion with a short derivation of \Cref{HCL:hardcore}.

\begin{proof}[Proof of \Cref{HCL:hardcore}]
  For each independent set $I \in \cI(\cH)$, set $g(I) \coloneqq S$ and $f(S) \coloneqq C$, where $(S,C)$ is the output of the algorithm with input~$I$.  Further, set $\cS \coloneqq \{ g(I) : I \in \cI(\cH) \}$.  Since the algorithm will perform the exact same operations when we replace the input set~$I$ with~$S$, the function $f$ is well-defined.
  Now, \Cref{lem:remain-ind-hardcore} implies assertion~\ref{item:contained-hardcore}, \Cref{lem:small-fingerprint-hardcore} proves assertion~\ref{item:small-fingerprint-hardcore}, and \Cref{lemma:tight-container-hardcore} establishes assertion~\ref{item:subexponential-probability}.
\end{proof}

\section{Derivations of the standard HCLs}
\label{sec:deducing-normal-containers}

\begin{proof}[Derivation of \Cref{HCL:efficient} from \Cref{HCL:cover}]
  We apply \Cref{HCL:cover} with $p \coloneqq \tau/(8r^2)$ to obtain a family $\cS$, functions $f$ and $g$, and a hypergraph $\cG_S$ for every $S \in \cS$.
  Assertions \ref{item:contained-efficient} and \ref{item:small-fingerprint-efficient} follow immediately from the respective assertions in \Cref{HCL:cover}, so we only need to argue that~\ref{item:small-cover-efficient} holds as well.

  To this end, fix some $S \in \cS$, let $C \coloneqq f(S)$, and let $\cG \coloneqq \cG_S$ be the hypergraph covering $\cH[C]$ that satisfies $w_p(\cG) \le p|C| \le p|V|$ and whose all edges have size at least two.
  By the assumption of the theorem, for every $E \in \cG$,
  \[
    \deg_{\cH}E \le K \cdot \left(\frac{p}{4K}\right)^{|E| - 1} \cdot \frac{e(\cH)}{|V|} < \frac{p^{|E|-1}}{2} \cdot \frac{e(\cH)}{|V|},
  \]
  where the second inequality follows as $|E| \ge 2$ and $K \ge r \ge 1$.
  The fact that $\cG$ covers $\cH[C]$ implies that
  \[
    e(\cH[C]) \le \sum_{E \in \cG} \deg_{\cH}E < \frac{e(\cH)}{2p|V|} \cdot \sum_{E \in \cG} p^{|E|} = \frac{e(\cH)}{2p|V|} \cdot w_p(\cH[C]) \le \frac{e(\cH)}{2}.
  \]
  On the other hand,
  \[
    e(\cH[C]) \ge e(\cH) - |V \setminus C| \cdot \Delta_1(\cH) \ge e(\cH) - |V \setminus C| \cdot \frac{Ke(\cH)}{|V|},
  \]
  which yields $|V \setminus C|/|V| > 1/(2K)$ or, equivalently, $|C| < \bigl(1-1/(2K)\bigr)|V|$.
\end{proof}

\begin{remark}
  A careful examination of the proof allows one to improve the bounds in assumption~\eqref{HCL:efficient:Delta:bounds} in \Cref{HCL:efficient} somewhat further.  First, note that we only needed the bounds 
  \[
    \Delta_1(\cH) \leq K \cdot \frac{e(\cH)}{|V|} \qquad \text{ and } \qquad \Delta_\ell(\cH) \leq \left( \frac{\tau}{32 r^2} \right)^{\ell -1} \, \frac{e(\cH)}{|V|} \quad \text{for $\ell \in \{2, \dotsc, r\}$}.
  \]
  
  Further, we claim that changing step~\ref{item:stopping-condition} in the algorithm from the proof of~\Cref{HCL:cover} to
  \begin{enumerate}
  \item[(2b')]
    If $w_p(\cH_i^{>1}) \le (\log r/r) \cdot p|C_i|$ or $|C_i| \le \bigl(1 - 1/(2K)\bigr)|V|$, then set $J \coloneqq i$ and \texttt{STOP}
  \end{enumerate}
  allows us to derive a stronger version of \Cref{HCL:efficient}, with the upper bound on $\Delta_{\ell}(\cH)$ in \eqref{HCL:efficient:Delta:bounds} increased by a factor of roughly $(r / \log r)^{\ell-1}$.
  To see this, notice that the only modifications needed in the proof of \Cref{HCL:cover} occur in \Cref{lem:largelink,lem:small-fingerprint-cover}.  The conclusion of \Cref{lem:largelink} still holds since
\[
    \frac{\log r}{r} \cdot p|C_i| \le w_p(\cH_i^{>1}) = \sum_{s=2}^r \sum_{v \in C_i} \frac{p \cdot w_p(\partial_v\cH_i^{s})}{s}
  \]
  and, consequently, there must be some $v \in C_i$ such that
  \[
    \sum_{s=2}^r \frac{w_p(\partial_v \cH_i^{s})}{s} \ge \frac{\log r}{r} \ge \sum_{s=2}^r \frac{1}{rs}
  \]
  which implies that $w_p(\partial_v \cH_i^{s})\geq 1/r$ for some $s$.
  On the other hand, the assertion of \Cref{lem:small-fingerprint-cover} may now be strengthened.  Indeed, since now
  \[
    w_p(\cH_J) = |V\setminus C_J| \cdot p + w_p(\cG) \le \frac{|V|p}{2K} + \frac{\log r}{r} \cdot |V| p \le \frac{2\log r}{r} \cdot |V|p,
  \]
  we may deduce the stronger bound  $|S| \leq 16(r \log r) \cdot p|V|$,
  which allows the conclusion of \Cref{HCL:efficient} to be reached under the weaker assumption
    \[
    \Delta_1(\cH) \leq K \cdot \frac{e(\cH)}{|V|} \qquad \text{ and } \qquad \Delta_\ell(\cH) \leq \left( \frac{\tau}{2^6 r\log r} \right)^{\ell -1} \frac{e(\cH)}{|V|} \quad \text{for $\ell \in \{2, \dotsc, r\}$}.
  \]
\end{remark}

\begin{proof}[Derivation of \Cref{HCL:efficient} from \Cref{HCL:hardcore}]
  We apply \Cref{HCL:hardcore} with $p \coloneqq \tau/(16Kr)$ and $\delta \coloneqq p/\tau = (16Kr)^{-1}$ to obtain a family $\cS$ and functions $f$ and $g$.
  Assertions~\ref{item:contained-efficient} and~\ref{item:small-fingerprint-efficient} follow immediately from the respective assertions in \Cref{HCL:hardcore}, so we only need to argue that~\ref{item:small-cover-efficient} holds as well.

  Suppose that $S \in \cS$ and let $C \coloneqq f(S)$.
  We may assume that $e(\cH[C]) \ge e(\cH)/2$, as otherwise $|C| \le \bigl(1-1/(2K)\bigr)|V|$ follows from the assumption that $\Delta_1(\cH) \le Ke(\cH)/|V|$, as in the previous derivation.
  Preparing to apply Janson's inequality (\Cref{thm:Janson}), we define $\mu \coloneqq \sum_{A \in \cH[C]} p^{|A|} = e(\cH[C])p^r$ and
  \[
    \Delta^* \coloneqq \sum_{A \in \cH[C]} p^{|A|} \sum_{\substack{B \in \cH[C] \\ A \cap B \neq \emptyset}} p^{|B \setminus A|} \le \mu \cdot \sum_{\ell = 1}^r \binom{r}{\ell}\Delta_\ell(\cH)p^{r-\ell}.
  \]
  By the assumption of the theorem,
  \[
    \sum_{\ell=1}^r \binom{r}{\ell}\Delta_\ell(\cH)p^{1-\ell} \le \sum_{\ell=1}^r r^\ell K\left(\frac{\tau}{32Kr^2p}\right)^{\ell-1} \cdot \frac{e(\cH)}{|V|} = \frac{Kre(\cH)}{|V|} \cdot \sum_{\ell=1}^r 2^{1-\ell},
  \]
  and consequently, by Janson's inequality,
  \[
    \Pr(C_p \in \cI(\cH)) \le \exp\left(-\frac{\mu^2}{2\Delta^*}\right) \le \exp\left(- \frac{\mu|V|}{4Kre(\cH)p^{r-1}}\right) \le \exp\left(-\frac{p|V|}{8Kr}\right),
  \]
  where the last inequality holds thanks to our assumption that $e(\cH[C]) \ge e(\cH)/2$.
  On the other hand, the assertion of \Cref{HCL:hardcore} gives
  \[
    \Pr\bigl(C_p \in \cI(\cH)\bigr) \ge \Pr\bigl(S \cup C_p \in \cI(\cH)\bigr) \ge (1-p)^{\delta|C|} > \exp\left(-2\delta p|V|\right) \ge \exp\left(-\frac{p|V|}{8Kr}\right),
  \]
  a contradiction.
\end{proof}

\begin{proof}[Derivation of \Cref{HCL:packaged} from \Cref{HCL:cover}]
  We apply \Cref{HCL:cover} to obtain a family $\cS$, functions $f$ and $g$, and a hypergraph $\cG_S$ for every $S \in \cS$.
  Assertions \ref{item:contained-packaged} and \ref{item:small-fingerprint-packaged} follow immediately from the respective assertions in \Cref{HCL:cover}, so we only need to argue that~\ref{item:non-supersaturated-packaged} holds as well.

  To this end, fix some $S \in \cS$, let $C \coloneqq f(S)$, and let $\cG \coloneqq \cG_S$ be the hypergraph covering $\cH[C]$ that satisfies $w_p(\cG) \le p|C|$ and whose all edges have size at least two.
  Assume toward contradiction that there is a hypergraph $\cH_S \subseteq \cH[C]$ that satisfies~\eqref{eq:p-supersaturated-def}.
  Since $\cG$ covers $\cH[C]$, and thus also $\cH_S$, we have
  \[
    e(\cH_S) \le \sum_{E \in \cG} \deg_{\cH_S}E < \sum_{E \in \cG} \frac{p^{|E|-1} e(\cH_S)}{|C|} = w_p(\cG) \cdot \frac{e(\cH_S)}{p|C|} \le e(\cH_S),
  \]
  a contradiction.
\end{proof}

\section{Interpolating between \Cref{HCL:cover,HCL:hardcore}}
\label{sec:interpolating-cover-hardcore}

In this section, we sketch the proof of \Cref{HCL:hardcore-cover}, which is a fairly straightforward adaptation of the proof of \Cref{HCL:hardcore} (given in \Cref{sec:proof-HCL-hardcore}), and present derivations of \Cref{HCL:cover,HCL:hardcore} from \Cref{HCL:hardcore-cover}.

\subsection{The algorithm}

Assume that $\cH$ is a hypergraph on a finite set $V$ and fix some $p \in (0,1]$.
The following algorithm takes as input an independent set $I \in \cI(\cH)$ and returns sets $S, C \subseteq V$ and a hypergraph $\cG$.

\begin{enumerate}[label=(\arabic*)]
\item
  Let $\cH_0 \coloneqq \cH$ and $S_0 \coloneqq \emptyset$.
\item
  Do the following for $i = 0, 1, \dotsc$:
  \begin{enumerate}[label=(\alph*),ref=(\arabic{enumi}\alph*)]
  \item
    \label{item:stopping-condition-hardcore-with-cover}
    If there exists a set $L \subseteq V \setminus S_i$ such that $L \in \cI(\cH_i)$ and
    \[
      \Pr\bigl(L \subseteq V_p \mid V_p \in \cI(\cH_i)\bigr) < (1-\delta)^{|L|}p^{|L|},
    \]
    let $L_i$ be a minimal such set;  otherwise, set $J \coloneqq i$ and \texttt{STOP}.
  \item
    If $L_i \subseteq I$, set $S_{i+1} \coloneqq S_i \cup L_i$ and $\cH_{i+1} \coloneqq \{E \setminus L_i:E\in \cH_i\} \eqqcolon \partial_{L_i}^*(\cH_i)$.
  \item
    Otherwise, if $L_i \nsubseteq I$, set $S_{i+1} \coloneqq S_i$ and $\cH_{i+1} \coloneqq \cH_i\cup \{L_i\}$.
  \end{enumerate}
\item  Return $S \coloneqq S_J$, $C \coloneqq \left\{v \in V : \{v\} \notin \cH_J\right\}$ and $\cG\coloneqq \cH_J[C]$.
\end{enumerate}

\subsection{Properties}

The above algorithm always terminates because no set $L$ is ever considered more than once.  Indeed, either $L_i \subseteq S_{i+1}$, and thus $L_i \nsubseteq V \setminus S_j$ for all $j > i$, or $L_i \in \cH_{i+1}$, and thus $L_i \notin \cI(\cH_j)$ for all $j > i$.  Moreover, the set $C$ and the hypergraph $\cG$ can be reconstructed with the knowledge of $S$ only, as running the algorithm with input $S$ instead of $I$ results in the same sets $S$ and $C$ and hypergraph $\cG$.  This allows us to define $g(I) \coloneqq S$ and $f(g(I)) \coloneqq C$.  It remains to show that, for every input $I \in \cI(\cH)$, we have $S \subseteq I \subseteq C$, the set S has at most $p|V|/\delta$ elements, and that $\cG$ is a cover of $\cH[C]$ with the desired property.

The assertion of \Cref{lem:remain-ind-hardcore}, that is, $S_i \subseteq I \in \cI(\cH_i)$ and $I' \in \cI(\cH_i) \iff S_i \cup I' \in \cI(\cH_i)$ for all $i$, remains true, with essentially the same proof.  Consequently, $I \in \cI(\cH_J)$ and thus $I \subseteq C$, as no vertex in $V \setminus C$ can belong to an independent set of $\cH_J$.  Observe additionally that, for all $i$, we have $\<\cH_i\> \subseteq \<\cH_{i+1}\>$ and thus $\cI(\cH_{i+1}) \subseteq \cI(\cH_i)$.  The desired upper bound on $|S|$ can be established by adapting the proof of \Cref{lem:small-fingerprint-hardcore}.  Indeed, note first that it is enough to show that the sequence $P_i \coloneqq \Pr\bigl(V_p \in \cI(\cH_i)\bigr)$ satisfies the inequality $P_{i+1} \le (1-\delta)^{\1_{L_i \subseteq I} \cdot |L_i|} \cdot P_i$ for all~$i$.  The inequality $P_{i+1} \le P_i$ follows as $\cI(\cH_{i+1}) \subseteq \cI(\cH_i)$.  Further, if $L_i \subseteq I$, then
\[
  p^{|L_i|} \cdot \Pr\bigl(V_p \in \cI(\cH_{i+1})\bigr) = \Pr\bigl(L_i \subseteq V_p \in \cI(\cH_i)\bigr) < (1-\delta)^{|L_i|} p^{|L_i|} \cdot \Pr\bigl(V_p \in \cI(\cH_i)\bigr),
\]
where the equality holds as $\cI(\cH_{i+1}) = \cI(\partial_{L_i}^* \cH_i) = \{I' \in \cI(\cH_i) : L_i \cup I' \in \cI(\cH_i)\}$.

Last but not least, we argue that $\cG$ possesses all the claimed properties.  First, $\cG = \cH_J[C]$ covers $\cH[C]$ since $\<\cH\> = \<\cH_0\> \subseteq \dotsb \subseteq \<\cH_J\>$.  Second, each $E \in \cG$ satisfies $|E| \ge 2$ as otherwise we would have $E \cap C = \emptyset$, by the definition of $C$.  We also remark that, for all $L \subseteq C \setminus S$ with $L \in \cI(\cG)$, we have
\[
  \Pr\bigl(L \subseteq C_p \mid C_p\in \cI(\cG)\bigr) = \Pr\bigl(L \subseteq V_p \mid V_p\in \cI(\cH_J)\bigr)\geq (1-\delta)^{|L|} p^{|L|},
\]
by the terminating condition in the algorithm and the fact that $\cI(\cH_J) \subseteq 2^C$.

Finally, we show that $E \cap S = \emptyset$ for every $E \in \cG$.  Since $S = S_J$ and $\cG \subseteq \cH_J$, it is enough to prove the stronger statement that $E \cap S_i = \emptyset$ for all $E \in \cH_i$ and all $i$.  We show this by induction on $i$.  The induction basis is vacuously true as $S_0 = \emptyset$.  Assume that $i \ge 0$.  If $L_i \nsubseteq I$, then $S_{i+1} = S_i$ and the only edge in $\cH_{i+1} \setminus \cH_i$ is $L_i$, which is disjoint from $S_i$.  Else, if $L_i \subseteq I$, then $S_{i+1} \setminus S_i = L_i$ and we replace each $E \in \cH_i$, which is disjoint from $S_i$, with the edge $E \setminus L_i = E \setminus S_{i+1}$, which is disjoint from $S_{i+1}$.

\subsection{Derivations of \Cref{HCL:cover,HCL:hardcore}}

We first show that Theorem~\ref{HCL:hardcore} follows by a simple application of Proposition~\ref{prop:key}.

\begin{proof}[Derivation of \Cref{HCL:hardcore} from \Cref{HCL:hardcore-cover}]
  Since the assumptions of the two theorems are identical, and their assertions differ only in the characterisation of the sets $f(S)$, we just need to show that assertion~\ref{item:description-containers-hardcore-cover} in~\Cref{HCL:hardcore-cover} implies that, for every $S \in \cS$, we have
  \[
    \Pr\bigl(S \cup C_p \in \cI(\cH)\bigr) \ge (1-p)^{\delta |C \setminus S|},
  \]
  where $C \coloneqq f(S)$.
  Let $C' \coloneqq C \setminus S$ and note that the fact that $\cG$ is a cover of $\cH[C]$ implies that
  \[
    \Pr\bigl(S \cup C_p \in \cI(\cH)\bigr) = \Pr\bigl(S \cup C_p' \in \cI(\cH)\bigr) \geq \Pr\bigl(S \cup C_p' \in \cI(\cG)\bigr).
  \]
  Further, the fact that $E \cap S = \emptyset$ for all $E \in \cG$ implies that, for every $v \in C'$
  \[
    \Pr\bigl(v \in C_p' \mid S \cup C_p' \in \cI(\cG)\bigr) = \Pr\bigl(v \in C_p' \mid C_p \in \cI(\cG)\bigr) \geq (1-\delta)p.
  \]
  It follows that
  \[
    \Ex\bigl[ |C_p'| \mid S \cup C_p' \in \cI(\cG)\bigr]\geq (1-\delta)p |C'|
  \]
  and consequently, by \Cref{prop:key}, that $\Pr\bigl(S \cup C_p'\in \cI(\cG)\bigr) \geq (1-p)^{\delta |C'|}$.
\end{proof}

The derivation of \Cref{HCL:cover} from \Cref{HCL:hardcore-cover} is decidedly more complicated.  We will show that that the cover $\cG$ of $\cH[C]$ given by \Cref{HCL:hardcore-cover} has the property  $w_p(\cG) \le p|C|$.  We will argue by contradiction, showing that, if $w_p(\cG) > p|C|$, some maximal set $L$ among those whose degree in $\cG$ is large satisfies $\Pr\bigl(L \subseteq C_p \mid C_p \in \cI(\cG)\bigr) < (1-\delta)^{|L|}p^{|L|}$; the proof of this inequality will employ a second-moment argument.

\begin{proof}[Derivation of \Cref{HCL:cover} from \Cref{HCL:hardcore-cover}]
  Suppose that $\cH$ is an $r$-uniform hypergraph with vertex set $V$ and let $p \in (0, 1/(8r^2)]$ be arbitrary.
  We apply \Cref{HCL:hardcore-cover} with $\delta \coloneqq 1/(4r)$ and density parameter $q \coloneqq 2rp \le 1/(4r) = \delta$ to obtain a family $\cS \subseteq 2^V$ of sets satisfying $|S| \le q|V|/\delta = 8r^2p|V|$ for each $S \in \cS$ and functions $g \colon \cI(\cH) \to \cS$ and $f \colon \cS \to 2^V$ such that $g(I) \subseteq I \subseteq f(g(I))$ for all $I \in \cI(\cH)$.
  It suffices to show that, for every $S \in \cS$, there is a hypergraph $\cG$ on $V$ with all edges of size at least two that covers $\cH[C]$ and satisfies $w_p(\cG) \le p|C|$, where $C \coloneqq f(S)$.
  To this end, fix some $S \in \cS$ and let $\cG$ be the set of all inclusion-minimal elements in the cover given by~\ref{item:description-containers-hardcore-cover} in \Cref{HCL:hardcore-cover}.  As $\cG$ covers $\cH[C]$ and satisfies $|E| \ge 2$ for all $E \in \cG$, we only need to show that $w_p(\cG) \le p|C|$.

  Assume by contradiction that $w_p(\cG) > p|C|$.  We claim that there exists a set $L \in \cI(\cG)$ and an $s \in \{2, \dotsc, r\}$ such that
  \begin{equation}
    \label{eq:hardcore-cover-heavy-L}
    w_p(\partial_L\cG^s) \ge 1/r.
  \end{equation}
  To this end, note that
  \[
    p|C| < w_p(\cG) = \sum_{s=2}^r w_p(\cG^s) = \sum_{s=2}^r \frac{1}{s} \sum_{v \in C} p \cdot w_p(\partial_v\cG^s) \le p \cdot \sum_{v \in C} \sum_{s=2}^r w_p(\partial_v\cG^s),
  \]
  which means that~\eqref{eq:hardcore-cover-heavy-L} must hold with $L = \{v\}$ for some $v \in C$ and $s \in \{2, \dotsc, r\}$; note that $\{v\} \in \cI(\cG)$ for every $v \in C$ since all edges of $\cG$ have at least two vertices.
  Let $L$ be an inclusion-maximal set that is independent in $\cG$ and satisfies~\eqref{eq:hardcore-cover-heavy-L} for some $s \in \{2, \dotsc, r\}$ and let $\ell \coloneqq |L| \in \br{s-1}$.  Such choice of $L$ guarantees that, on the one hand,
  \[
    |\partial_L\cG^s| = p^{\ell-s} \cdot w_p(\partial_L \cG^s) \ge p^{\ell-s} / r
  \]
  and, on the other hand, for every $t \in \br{s-\ell-1}$,
  \[
    \begin{split}
      \Delta_t(\partial_L\cG^s) & = \max\{|\partial_{L \cup T}\cG^s| : T \subseteq V \setminus L \wedge |T|=t\} \\
      & = \max\{p^{t+\ell-s} \cdot w_p(\partial_{L \cup T}\cG^s) : T \subseteq V \setminus L \wedge |T|=t\} < p^{t+\ell-s} / r,
    \end{split}
  \]
  where the last inequality follows as every set $L' \supseteq L$ that has fewer than $s$ elements and satisfies $\partial_{L'}\cG^s \neq \emptyset$ is independent in $\cG$ (otherwise, $L'$ would be a proper subset of some edge of $\cG^s$, contradicting the fact that $\cG$ is an antichain).

  Since $L \in \cI(\cG)$, assertion~\ref{item:description-containers-hardcore-cover} of \Cref{HCL:hardcore-cover} yields
  \[
    \Pr\bigl(L \subseteq C_q \mid C_q \in \cI(\cG)\bigr) \ge (1-\delta)^\ell q^\ell \ge (1-\delta\ell) q^\ell \ge 3q^\ell/4
  \]
  whereas, by Harris's inequality, writing $\tC$ for $C_q$ conditioned on the event $C_q \in \cI(\cG)$,
  \[
    \begin{split}
      \Pr\bigl(L \subseteq C_q \mid C_q\in \cI(\cG)\bigr) & = \Pr\bigl(L \subseteq C_q \wedge C_q \in \cI(\partial_L \cG) \mid C_q \in \cI(\cG) \bigr) \\
      & \le q^{\ell} \cdot \Pr\bigl(C_q\in \cI(\partial_L\cG) \mid C_q\in \cI(\cG)\bigr) = q^{\ell} \cdot \Pr\bigl(\tC \in \cI(\partial_L\cG)\bigr),
    \end{split}
  \]

  Let $X \coloneqq e(\partial_L\cG^s[\tC])$ denote the number of edges that $\tC$ induces in $\partial_L\cG^s$.  Using the Paley--Zygmund inequality, we may bound
  \[
    \frac{3}{4} \le \Pr\bigl(\tC \in \cI(\partial_L\cG)\bigr) \leq \Pr\bigl(\tC \in \cI(\partial_L\cG^s)\bigr) = \Pr(X = 0) \le 1 - \frac{\Ex[X]^2}{\Ex[X^2]},
  \]
  which means that $\Ex[X^2] \ge 4\Ex[X]^2$.

  Since each $E \in \partial_L\cG^s$ is independent in $\cG$ (as $L \neq \emptyset$ and $\cG$ is an antichain), assertion~\ref{item:description-containers-hardcore-cover} of \Cref{HCL:hardcore-cover} yields
  \[
    \Ex[X] = \sum_{E \in \partial_L\cG^s} \Pr(E \subseteq \tC) \ge \sum_{E \in \partial_L\cG^s} (1-\delta)^{|E|}q^{|E|} = (1-\delta)^{s-\ell} \cdot \underbrace{|\partial_L\cG^s| \cdot q^{s-\ell}}_{\mu}.
  \]
  Further,
  \[
    \mu = |\partial_L\cG^s| \cdot q^{s-\ell} = w_p(\partial_L\cG^s) \cdot \left(\frac{q}{p}\right)^{s-\ell} \ge \frac{1}{r} \cdot \left(\frac{q}{p}\right)^{s-\ell} \ge 2.
  \]
  On the other hand, by Harris's inequality,
  \[
    \begin{split}
      \Ex[X^2]
      & = \sum_{E, E' \in \partial_L\cG^s} \Pr(E \cup E' \subseteq \tC) \le \sum_{E, E' \in \partial_L\cG^s} q^{|E \cup E'|}
        = \sum_{E \in \partial_L\cG^s} q^{|E|} \sum_{E' \in \partial_L\cG^s} q^{|E' \setminus E|} \\
      & \le \mu \cdot \left(\mu + 1 + \sum_{t=1}^{s-\ell-1} \binom{s-\ell}{t} \Delta_t(\partial_L\cG^s) q^{s-\ell-t} \right).
    \end{split}
  \]
  Further, for every $t \in \br{s-\ell-1}$,
  \[
    \Delta_t(\partial_L\cG^s) \cdot q^{s-\ell-t} \le \frac{p^{t+\ell-s} \cdot q^{s-\ell-t}}{r} \le \left(\frac{p}{q}\right)^t \cdot \mu = \frac{\mu}{(2r)^t}
  \]
  and thus
  \[
    \Ex[X^2] \le \mu + \mu^2 \cdot \sum_{t=0}^{s-\ell-1} \binom{s-\ell}{t} \cdot \frac{1}{(8r)^t} \le \mu + \mu^2 \cdot \left(1 + \frac{1}{2r}\right)^{s-\ell} \le \mu + \frac{3}{2} \cdot \mu^2 \le 2 \mu^2.
  \]
  We conclude that
  \[
    \frac{\Ex[X^2]}{\Ex[X]^2} \le \frac{2}{(1-\delta)^{s-\ell}} \le \frac{2}{1-\delta r} \le \frac{8}{3},
  \]
  which is a contradiction.
\end{proof}

\section{Concluding remarks}
\label{sec:concluding-remarks}

Comparing the two notions of almost-independence that are used to describe the containers in \Cref{HCL:cover,HCL:hardcore} led us to proving \Cref{prop:prob-vs-covers}.  While part~\ref{item:prob-vs-covers-2} of this proposition is sufficient for our purposes, we hope that a much stronger statement is true.  We will say that a~hypergraph $\cH$ is \emph{$r$-bounded}, for some positive integer $r$, if $|E| \le r$ for all $E \in \cH$.

\begin{question}\label{q:efficient-cover}
  What is the smallest function $K \colon \mathbb{N} \to \mathbb{R}$ such that any $r$-bounded hypergraph $\cH$ on a finite vertex set $V$ admits a cover $\cG \subseteq 2^V \setminus \{\emptyset\}$ that satisfies
  \[
    \Pr(\cH[V_p]=\emptyset) \leq \exp\left(-w_{p/K(r)}(\cG)\right)?
  \]
\end{question}

\Cref{prop:prob-vs-covers}~\ref{item:prob-vs-covers-2} shows that $K(r) = O(r^2)$ when one assumes that $\cH$ is $r$-uniform rather than $r$-bounded.
However, a fairly straightforward adaptation of the proof of this proposition extends its validity to all $r$-bounded hypergraphs.
On the other hand, it is not hard to see that $K(r) = \Omega(\log r)$.
(Indeed, consider the hypergraph $\cH$ with vertex set $K_{2n}$ whose edges are all perfect matchings and $p = c\log n / n$ for some small constant $c$.)
Originally, we conjectured that $K$ does not have to depend on $r$, at the cost of multiplying the above upper bound on the probability by a factor polynomial in $r$.
Soon after the original version of this work was made public, Quentin Dubroff found a counterexample to this stronger conjecture.
(His counterexample is the hypergraph with vertex set $K_n$ whose edges are all $K_{1,d}$-factors in $K_n$, for some $d = n^{\Theta(1)}$, and $p$ slightly larger than $d/n$.)
We later realised that a universal upper bound on the probability that $\cH[V_p] = \emptyset$ that has the form $\exp(-w_{p/K(r)}(\cG) + f(r))$, for some $f(r) \ge 0$ that depends only on $r$, implies the corresponding stronger upper bound of $\exp(-w_{p/K(r)}(\cG))$.
To see this, consider the hypergraph $m\cH$ that comprises $m$ vertex-disjoint copies of $\cH$.
It is easy to verify that $\Pr(m\cH[mV_p] = \emptyset) = \Pr(\cH[V_p] = \emptyset)^m$ and that, for every $q$, the smallest $q$-weight of a cover of $m\cH$ is $m$ times the smallest $q$-weight of a cover for $\cH$.
This observation leads to many further counterexamples to our (very naive) conjecture.

We remark that showing that $K(r)=O(\log r)$ would imply the Kahn--Kalai Conjecture / Park--Pham Theorem~\cite[Theorem~1.1]{ParPha24}; it would also be consistent with the strengthening of the Park--Pham Theorem due to Bell~\cite{Bel23}.  Indeed, let $V$ be a finite set and let $\cF \subseteq 2^V$ be an arbitrary increasing property.  Recall that the \emph{expectation threshold} of $\cF$ is the number $q(\cF)$ defined as follows:
\[
  q(\cF) \coloneqq \max\{q : \exists\, \cG \subseteq 2^V \; \<\cG\> \supseteq \cF \wedge w_q(\cG) \leq 1/2\}.
\]
Let $\cH$ be the set of all minimal elements of $\cF$ and let $r \coloneqq \max\{|E| : E \in \cH\} \cup \{2\}$, so that $\cH$ is an $r$-bounded hypergraph on $V$ and $R \in \cF$ if and only if $\cH[R] \neq \emptyset$.
Let $q \coloneqq q(\cH)$, suppose that $p \ge 4K(r) \cdot q$, and let $\cG$ be the cover of $\cH$ from the statement of \Cref{q:efficient-cover}.  Observe that
\[
  w_{p/K(r)}(\cG) \ge p/(qK(r))\cdot w_q(\cG) \ge 4 \cdot w_q(\cG) \ge 2.
\]
Consequently, we may conclude that
\[
  \Pr(V_p \notin \cF) = \Pr(\cH[V_p] = \emptyset)  \le 2/e^2 \le 1/2.
\]

Finally, showing that $K(r) = o(r)$ would imply an improvement of the upper bound on the number of containers in \Cref{HCL:cover}, as the following theorem shows.

\begin{HCL}
  \label{HCL:conjecture}
  Let $\cH$ be an $r$-bounded hypergraph with a finite vertex set $V$.  For all reals $\delta$ and $p$ that satisfy $0< p \le \delta/(2K(r)^2)$, there exists a family $\cS \subseteq 2^V$ and functions
  \[
    g \colon \cI(\cH)\to \cS
    \qquad \text{and} \qquad
    f \colon \cS\to 2^{V}
  \]
  such that:
  \begin{enumerate}[label=(\alph*)]
  \item
    For each $I\in \cI(\cH)$, we have $g(I) \subseteq I \subseteq f(g(I))$.
  \item
    Each $S \in \cS$ has at most $2K(r)^2 p |V|/\delta$ elements.
  \item
    For every $S \in \cS$, letting $C \coloneqq f(S)$, there exists a hypergraph $\cG \subseteq 2^{C \setminus S} \setminus \{\emptyset\}$ with
    \[
      w_{p}(\cG) \le \delta p|C|
    \]
    that covers $\cH[C]$.
  \end{enumerate}
\end{HCL}
\begin{proof}
  Let $\cS$, $f$, and $g$ be the collection and the functions given by \Cref{HCL:hardcore} with $K(r) p$ in place of $p$ and $\delta/(2K(r))$ in place of $\delta$.
  Fix some $S \in \cS$, let $C \coloneqq f(S)$, and let $\cG \subseteq 2^{C \setminus S} \setminus \{\emptyset\}$ be a cover of $\cH[C]$ with smallest $p$-weight.
  The assertion of \Cref{HCL:hardcore} and the definition of $K(r)$ yield the following two inequalities:
  \[
    \bigl(1-K(r)p\bigr)^{\delta|C \setminus S|/2K(r)} \le \Pr\bigl(S \cup C_{K(r)p} \in \cI(\cH)\bigr) \le \exp\bigl(- w_p(\cG)\bigr).
  \]
  Taking logarithms of both sides yields the inequality
  \[
    w_p(\cG) \le -\frac{\delta|C \setminus S|}{2K(r)} \cdot \log\bigl(1-K(r)p\bigr) \le \delta p |C|,
  \]
  as required.  
\end{proof}

\section{Acknowledgements}
First and foremost, we are indebted to Quentin Dubroff for supplying a counterexample to a conjecture stated in a previous version of this article.  Further, we would like to thank Jinyoung Park for helpful discussions of this conjecture and the counterexample.  Last but not least, we would like to thank Rob Morris, Peter Keevash, and Huy Pham for helpful comments and suggestions.

\bibliographystyle{abbrv}
\bibliography{containers}

\appendix

\section{Three proofs of \Cref{prop:key}}
\label{sec:three-proofs-key-prop}

One can see the assertion of the proposition is best-possible by picking an arbitrary subset $U \subseteq C$ and considering the family $\cI=2^{U}$.  Indeed, in this case $\Ex[|C_p| \mid C_p \in \cI]=\Ex[|U_p|]=p|U|$ and
\[
  \log \Pr(C_p\in \cI)=\log \Pr(C_p\cap (C\setminus U)=\emptyset)=(|C|-|U|)\log(1-p).
\]

We provide three proofs of the proposition.   The first proof uses the chain rule for relative entropy, the second proof employs a compression argument combined with the Kruskal--Katona theorem, and the third proof uses a generalisation of the edge-isoperimetric inequality for the hypercube due to Kahn and Kalai~\cite{KahKal07}.
\begin{proof}[First proof]
  Without loss of generality, we may assume that $C = \br{N}$.
  Let $Y = (Y_1, \dotsc, Y_N)$ be the indicator function of $C_p$ conditioned on $C_p \in \cI$ and note that
  \[
    \log \Pr(C_p\in \cI)= -\DKL(Y\, \| \,C_p),
  \]
  where $\DKL(\cdot \, \| \, \cdot)$ denotes the \emph{Kullback--Leibler divergence} between two probability distributions.
  For any $J \subseteq \br{N}$, write $I_p((Y_j)_{j \in J})$ in place of $\DKL((Y_j)_{j \in J} \,\|\, J_p)$.
  With this notational convention, $\DKL(Y \,\|\, C_p) = I_p(Y_1, \dotsc, Y_N)$.
  The first key property of $I_p$ that we will use is that it obeys the familiar chain rule (see, e.g., \cite[Section~4]{KozSam23}):
  \[
    I_p(Y_1, \dots, Y_N) = \sum_{i=1}^N I_p(Y_i \mid Y_1, \dotsc, Y_{i-1}) = \sum_{i=1}^N \Ex[i_p(\Ex[Y_i \mid Y_1, \dotsc, Y_{i-1}])],
  \]
  where we wrote $I_p(Y_i \mid Y_1, \dotsc, Y_{i-1})$ for the conditional relative entropy and
  \[
    i_p(q) \coloneqq \DKL(\Ber(q) \,\|\, \Ber(p)) = q\log\frac{q}{p} + (1-q)\log\frac{1-q}{1-p}.
  \]

  
  Since the function $i_p \colon [0,1] \to \RR$ is convex, we have, for every random variable $Z \in [0,p]$,
  \[
    \Ex[i_p(Z)] \le (1-\Ex[Z]/p) \cdot i_p(0) + (\Ex[Z]/p) \cdot i_p(p) = (\Ex[Z]/p-1) \log(1-p).
  \]
  Finally, since conditioned on $Y_1, \dotsc, Y_{i-1}$, the vector $(Y_i, \dotsc, Y_N)$ is the indicator function of $(\br{N} \setminus \{1, \dotsc, i-1\})_p$ conditioned on belonging to a decreasing family of sets, we have $\Ex[Y_i \mid Y_1, \dotsc, Y_{i-1}] \le p$ almost surely.  This implies that
  \[
    I_p(Y_1, \dotsc, Y_N) \le \sum_{i=1}^N (\Ex[Y_i]/p-1) \log(1-p) = (\Ex[|Y|]/p - N) \log(1-p),
  \]
  as claimed.
\end{proof}

\newcommand{\bx}{\mathbf{x}}

\begin{proof}[Second proof]
  We are going to prove the following equivalent inequality:
  \begin{equation}
    \label{eq:Ex-log-Pr-restated}
    \Ex[|C_p| \mid C_p \in \cI] \le \bigl(|C| - \log_{1-p} \Pr(C_p \in \cI)\bigr) \cdot p.
  \end{equation}
  Let $\ell = \ell(p) \coloneqq \log_{1-p} \Pr(C_p \in \cI)$.
  We first argue that it is enough to establish~\eqref{eq:Ex-log-Pr-restated} in the case where $\ell$ is an integer.
  To see this, suppose that $\ell \notin \Z$.
  We claim that the (continuous) function $q \mapsto \ell(q)$ cannot be constant on any open interval $I$ containing $p$ and thus it takes rational values for $q$ arbitrarily close to $p$.
  Suppose that this were not true and $\Pr(C_q \in \cI) = (1-q)^\ell$ for all $q \in I$.
  We would then have, for all $m$,
  \[
    \frac{d^m}{dq^m} \Pr(C_q \in \cI) = \prod_{i=0}^{m-1} (i-\ell) \cdot (1-q)^{\ell-m} \neq 0,
  \]
  contradicting the fact that $\Pr(C_q \in \cI) = \sum_{I \in \cI} q^{|I|}(1-q)^{|C|-|I|}$ is a polynomial in $q$.
  Since both sides of~\eqref{eq:Ex-log-Pr-restated} are continuous in $p$, it is enough to prove this inequality when $\ell$ is rational.
  
  Assume now that $\ell \in \Q$ and let $b$ be a positive integer such that $b\ell \in \Z$.
  Consider the family $\cI^b \subseteq 2^{C \times \br{b}}$ defined by
  \[
    \cI^b \coloneqq \bigl\{I_1 \times \{1\} \cup \dotsb \cup I_b \times \{b\} : I_1, \dotsc, I_b \in \cI\bigr\}.
  \]
  and observe that $\cI^b$ is decreasing and that $\Pr( (C \times \br{b})_p \in \cI^b) = \Pr(C_p \in \cI)^b = (1-p)^{b\ell}$.
  Invoking~\eqref{eq:Ex-log-Pr-restated} with $\cI$ replaced by $\cI^b$, we conclude that
  \[
    b \cdot \Ex[|C_p| \mid C_p \in \cI] = \Ex\bigl[|(C \times \br{b})_p| \in \cI^b \mid (C \times \br{b})_p \in \cI^b|\bigr]
    \le \bigl(b|C| - b\ell\bigr) \cdot p.
  \]

  Assume $\ell \in \Z$, let $\lambda \coloneqq p/(1-p)$, and let $Z \coloneqq \sum_{I \in \cI} \lambda^{|I|}$.
  Our first key observation is that
  \[
    Z = (1-p)^{-|C|} \cdot \Pr(C_p \in \cI) = (1-p)^{\ell-|C|} = (1+\lambda)^{|C|-\ell}.
  \]
  Write $n \coloneqq |C| - \ell$, so that $Z = (1+\lambda)^n$.
  For each $i \in \br{n}$, let $x_i$ be the unique number in $\{0\} \cup [i, \infty)$ such that $\binom{x_i}{i} = |\binom{C}{i} \cap \cI|$.
  The assumption that $\cI$ is decreasing implies that $2^I \subseteq \cI$ for all $I \in \cI$ and, consequently, $Z \ge (1+\lambda)^{|I|}$ for each $I \in \cI$.
  Therefore, $\max_{I \in \cI}|I| \le n$ and
  \[
    f(x_1, \dotsc, x_n) \coloneqq 1 + \sum_{i=1}^n \binom{x_i}{i} \cdot \lambda^i = \sum_{i=0}^{|C|} \left|\binom{C}{i} \cap \cI\right| \cdot \lambda^i = Z.
  \]
  Our second key observation is that
  \[
    g(x_1, \dotsc, x_n) \coloneqq \sum_{i=1}^n \binom{x_i}{i} \cdot i\lambda^i = Z \cdot \Ex[|C_p| \mid C_p \in \cI].
  \]
  Since $\cI$ is decreasing, the Kruskal--Katona theorem implies that $x_1 \ge \dotsb \ge x_n$.  In particular, letting
  \[
    \cX \coloneqq \{(x_1, \dotsc, x_n) \in \R^n : x_1 \ge \dotsb \ge x_n \text{ and } x_i \in \{0\} \cup (i-1, \infty) \text{ for all $i \in \br{n}$}\},
  \]
  we have
  \[
    \Ex[|C_p| \mid C_p \in \cI] \le Z^{-1} \cdot \max\{g(x_1, \dotsc, x_n) : (x_1, \dotsc, x_n) \in \cX \wedge f(x_1, \dotsc, x_n) = Z\}.
  \]
  Let $\bx = (x_1, \dotsc, x_n)$ be a vector that achieves the above maximum.
  We claim that $x_1 = \dotsb = x_n$.
  Indeed, suppose that $x_i > x_{i+1}$ for some $i$ and let $x'$ be the unique number satisfying $\max\{x_{i+1}, i\} < x' < x_i$ and
  \[
    d \coloneqq \binom{x_i}{i} - \binom{x'}{i} = \left[\binom{x'}{i+1} - \binom{x_{i+1}}{i+1}\right] \cdot \lambda.
  \]
  Let $\bx'$ be the vector obtained from $\bx$ by replacing the $i$th and the $(i+1)$th coordinates by $x'$
  and note that
  \[
    f(\bx') - f(\bx) = \left[\binom{x'}{i} - \binom{x_i}{i}\right] \cdot \lambda^i + \left[\binom{x'}{i+1} - \binom{x_{i+1}}{i+1}\right] \cdot \lambda^{i+1} = - d\lambda^i + d \lambda^i = 0
  \]
  and, similarly,
  \[
    g(\bx') - g(\bx) = (i+1)d\lambda^i - id\lambda^i = d\lambda^i > 0,
  \]
  contradicting the maximality of $\bx$.
  Let $x \in (n-1, \infty)$ be the number such that $x_1 = \dotsb = x_n = x$.
  Since
  \[
    (1+\lambda)^n = Z = 1 + \sum_{i=1}^n \binom{x}{i} \cdot \lambda^i = \sum_{i=1}^n \binom{x}{i} \cdot \lambda^i,
  \]
  we have $x = n$ and thus
  \[
    \Ex[|C_p| \mid C_p \in \cI] \le (1+\lambda)^{-n} \cdot \sum_{i=1}^n \binom{n}{i} \cdot i\lambda^i = n \cdot \frac{\lambda}{1+\lambda} = np, 
  \]
  as claimed.
\end{proof}

\begin{proof}[Third proof]
  Let $q \coloneqq 1-p$ and let $f \colon 2^C \to \{0,1\}$ be the function defined by $f(A) = 1$ if and only if $A^c \in \cI$.  Note that $f$ is increasing and that $\Ex[f(C_q)] = \Pr(C_p \in \cI)$.
  Recall that the \emph{total influence} of $f$ with respect to $C_q$ is the quantity
  \[
    \Inf_q(f)
    \coloneqq \sum_{x \in C} \Pr\bigl(f(C_q \cup \{x\}) \neq f(C_q \setminus \{x\})\bigr)
    = \sum_{x \in C} \bigl(\Pr(C_p \setminus \{x\} \in \cI) - \Pr(C_p \cup \{x\} \in \cI)\bigr).
  \]
  Observe now that, for every $x \in C$,
  \[
    \Pr(C_p \setminus \{x\} \in \cI) = \Pr(C_p \in \cI \mid x \notin C_p) = \frac{\Pr(C_p \in \cI)}{\Pr(x \notin C_p)} \cdot \Pr(x \notin C_p \mid C_p \in \cI)
  \]
  and, similarly,
  \[
    \Pr(C_p \cup \{x\} \in \cI) = \Pr(C_p \in \cI \mid x \in C_p) = \frac{\Pr(C_p \in \cI)}{\Pr(x \in C_p)} \cdot \Pr(x \in C_p \mid C_p \in \cI).
  \]
  Consequently, letting $p_x \coloneqq \Pr(x \in C_p \mid C_p \in \cI)$, we have
  \[
    \frac{\Inf_q(f)}{\Pr(C_p \in \cI)} = \sum_{x \in C} \left(\frac{1-p_x}{1-p} - \frac{p_x}{p}\right) = \frac{1}{1-p} \cdot \sum_{x \in C} \left(1 - \frac{p_x}{p}\right) = \frac{|C| - \Ex[|C_p| \mid C_p \in \cI] / p}{1-p}.
  \]
  The desired inequality follows from the edge-isoperimetric inequality for $f$, which states that
  \[
    q \cdot \Inf_q(f) \ge \Ex[f(C_q)] \cdot \log_q \Ex[f(C_q)],
  \]
  see~\cite[Section~4]{KahKal07}.
\end{proof}

\end{document}